\documentclass[10pt,psamsfonts]{amsart}
\usepackage{amsmath}
\usepackage{amsthm}
\usepackage{amssymb}
\usepackage{amscd}
\usepackage{amsfonts}
\usepackage{amsbsy}
\usepackage{graphicx}
\usepackage[dvips]{psfrag}
\usepackage{array}
\usepackage{color}
\usepackage{epsfig}
\usepackage{url}
\usepackage{overpic}
\usepackage{tikz-cd}
\usepackage{enumitem}
\usepackage{hyperref}
\usepackage{soul}
\usepackage{float}
\usepackage[normalem]{ulem}
\usepackage{autobreak}
\usepackage{crossreftools}
\usepackage{bm}
\usepackage{cleveref}

\usepackage{textcomp}

\newcommand{\R}{\ensuremath{\mathbb{R}}}
\newcommand{\N}{\ensuremath{\mathbb{N}}}

\renewcommand{\Im}{{\rm Im~}}

\newcommand{\Id}{\mathrm{Id}}

\newcommand{\sgn}{\mathrm{sign}}

\DeclareMathOperator\arctanh{arctanh}

\def\p{\partial}
\def\e{\varepsilon}

\newtheorem {theorem} {Theorem}
\newtheorem {definition} {Definition}
\newtheorem {proposition}{Proposition}
\newtheorem {corollary}{Corollary}
\newtheorem {lemma}{Lemma}

\newtheorem {remark}{Remark}

\crefname{equation}{}{}
\Crefname{equation}{}{}
\crefname{section}{Section}{}
\crefname{proposition}{Proposition}{}
\crefname{theorem}{Theorem}{Theorems}
\crefname{lemma}{Lemma}{Lemmas}
\crefname{corollary}{Corollary}{Corollaries}
\crefname{definition}{Definition}{Definitions}

\usepackage{mathpazo}
\usepackage[foot]{amsaddr}

\begin{document}
\renewcommand{\arraystretch}{1.5}

\title{Averaging theory and catastrophes}
\author[Pedro C.C.R. Pereira, Mike R. Jeffrey, and Douglas D. Novaes]
{Pedro C.C.R. Pereira$^1$, Mike R. Jeffrey$^2$, and Douglas D. Novaes$^3$}

\address{School of Engineering Mathematics and Technology, University of Bristol, Ada Lovelace Building, Bristol BS8 1TW, UK$^2$}
\address{Universidade Estadual de Campinas (UNICAMP), Departamento de Matem\'{a}tica, Instituto de Matemática, Estatística e Computação Científica (IMECC) - Rua S\'{e}rgio Buarque de Holanda, 651, Cidade
Universit\'{a}ria Zeferino Vaz, 13083--859, Campinas, SP, Brazil$^{1,3}$}
\email{pedro.pereira@ime.unicamp.br$^1$}
\email{mike.jeffrey@bristol.ac.uk$^2$}
\email{ddnovaes@unicamp.br$^3$}

\keywords{Averaging theory, ordinary differential equations, bifurcation theory, catastrophe theory}

\subjclass[2020]{34C29, 37G10, 37G15, 58K35}

\begin{abstract}
When a dynamical system is subject to a periodic perturbation, the averaging method can be applied to obtain an autonomous leading order `guiding system', placing the time dependence at higher orders. Recent research focused on investigating invariant structures in non-autonomous differential systems arising from hyperbolic structures in the guiding system, such as periodic orbits and invariant tori. Complementarily, the effect that bifurcations in the guiding system have on the original non-autonomous one has also been recently explored, albeit less frequently. This paper extends this study by providing a broader description of the dynamics that can emerge from non-hyperbolic structures of the guiding system. Specifically, we prove here that $\mathcal{K}$-universal bifurcations in the guiding system `persist' in the original non-autonomous one, while non-versal bifurcations, such as the transcritical and pitchfork, do not. We illustrate the results on examples of a fold, a transcritical, a pitchfork, and a saddle-focus.
\end{abstract}

\maketitle

\section{Introduction}\label{sec:intro}

The method of averaging allows time-dependent singular perturbations of autonomous dynamical systems to be moved to higher orders in the perturbation parameter. The leading order term, sometimes called the {\it guiding} system, is then time-independent, but captures the average of the time-varying perturbation. Typically, if the perturbation is periodic, one can then show that equilibria of the guiding system constitute periodic orbits in the full system, see e.g. \cite{Guckenheimer1983,Sanders2007} and with more generality \cite{LliNovTei2014,NS21}. Other invariant structures have also been studied, for example periodic orbits of the time-independent system leading to invariant tori \cite{NP24,PEREIRA20231}. To describe what happens when bifurcations occur in the guiding system is a harder problem, and only solved for limited cases, such as a fold or Hopf bifurcation under certain conditions in \cite{CANDIDO20204555,Guckenheimer1983}. Bifurcation theory itself cannot generally be directly applied because of the singular nature of the systems. In this paper we provide the necessary theory to study whether bifurcations `persist' under averaging.

In essence here we will study systems of the form $\dot X=\e f(t,X,\mu,\e)$, where $\mu \in \R^k$ is some parameter. Such equations describe the effect of a time-varying perturbation $\e f(t,X,\mu,\e)$ near time-independent invariants, e.g. near the equilibrium $h=0$ of an autonomous system $\dot X =h(X)$, and these occur commonly, for example, in studying small perturbations of oscillators, particularly using Melnikov methods (see e.g. \cite{Guckenheimer1983,Sanders2007}). The quantity $X$ then varies on the timescale $t/\e$, making the perturbation singular over non-vanishing intervals of time $[0,t]$.
	
However, according to standard results of averaging theory \cite{Sanders2007}, after change of variables, we can write this for some $l\in\mathbb N$ as
\begin{equation} \label{eq:averagedsystemintro}
	\dot x=\e^\ell g_\ell(x,\mu)+\e^{\ell+1}R_\ell(t,x,\mu,\e),
\end{equation}
where the leading order of the perturbation is given as a regular autonomous perturbation. Through a time rescaling this becomes $\dot x = g_\ell(x,\mu) + \e R_\ell(t/\e^\ell,x,\mu,\e)$, so that time-dependence enters only as a perturbation (albeit singular) of an otherwise autonomous system. Here $f,g_\ell$, and $R_\ell$ are differentiable functions we will specify more completely later.

What happens when such a system undergoes a bifurcation has received relatively little attention. 
In \cite{Guckenheimer1983} it is shown that certain one-parameter bifurcations (a fold or a Hopf) of $\dot x = \e g_1(x,\mu)$ persist after being perturbed as above. Somewhat subtly, these results actually prove the existence of branches of equilibria or cycles around their bifurcation points, not of the bifurcations themselves, and moreover assume certain forms of system that are not entirely general. In \cite{CANDIDO20204555}, it is shown in more detail that a Hopf bifurcation in $\dot x = \e^\ell g_\ell(x,\mu)$ persists as a Neimark-Sacker bifurcation in the time-$T$ map of the original non-autonomous differential system, creating invariant tori.

Compared to bifurcations in averaging theory, the literature on other aspects of time-dependent perturbations of autonomous systems is extensive. Such investigations have been around since Poincar\'e's study of systems of the form $\dot u=g(u)+\e h(u,t,\e)$, from which are derived the origins of homoclinic tangles and chaos \cite{Poincare1890,Poincare1899,HOLMES1990137}. 
They remain novel today in multi-variable and multi-timescale problem, notably in models of neuron bursting via mixed-mode oscillations, see e.g. \cite{Krupa2008}. A simpler example is the singularly perturbed pendulum, $\ddot u=-\sin(u)+\e\sin(t/\e)$, e.g. \cite{Holmes1988}.

Here we will show indeed that a broad class of bifurcations of the leading order guiding system $\dot x = g_\ell(x,\mu)$ (in any number of parameters) `persists' when carried over to the original time-dependent system. 
We will use an idea from \cite{j22cat,j23bgconditions,j24cat} of looking only at the {\it catastrophe} underlying any bifurcation, which considers only the numbers of equilibria involved in a bifurcation (so-called $\mathcal{K}$-equivalence), taking no interest in topological equivalence classes. This provides an essential simplification making it possible to prove classes of bifurcation that do or do not `persist' under averaging.

Our interest will particularly be in families whose guiding systems exhibit non-hyperbolic equilibria that induce bifurcations, so that the rescaled form of \cref{eq:averagedsystemintro} can be, for instance, a simple fold point with a small singular time-dependent perturbation, written as $\dot x = x^2 + \e h(t/\e)$. 
A notable application of this is to seasonal differential models, i.e., differential systems with time varying parameters. For example, consider the family $\dot u=u^2 + p$ depending on a parameter $p$. What happens if, in fact, $p$ undergoes small time fluctuations? We will show that the average value $\mu_p$ of $p(t)$ can play the role of a bifurcation parameter whose variation precipitates catastrophes of periodic solutions.

We will also show that bifurcations that are not stable, except under restrictions such as symmetries important in applications, may nevertheless form stable systems under averaging. As an illustration, we will consider systems with transcritical and pitchfork bifurcations, as those appear frequently in the literature. We will show how, if one of these non-generic bifurcations appears in the guiding system, the addition of a time-varying singular perturbation generically produces stable bifurcations of periodic solutions in the averaged system.

The paper is arranged as follows: in \Cref{sec:results}, we present an overview of our main results, written as a practical summary for the non-specialist in either averaging or singularities and catastrophes. We add to this in \Cref{sec:examples} by illustrating applications to some simple examples, and particularly to the physical application of systems with time-varying parameters. The remainder of the paper contains the results presented more formally: in \Cref{sec:prelim}, we introduce known concepts extracted from both singularity theory and the averaging method that are needed to discuss and prove our results, before proving our main results in \Cref{sec:proofs}, and collecting a few auxiliary results of a more technical nature.


\section{Overview of results} \label{sec:results}

Consider a $(k+1)$-parameter family of $n$-dimensional systems in the form
\begin{equation} \label{eq:standardsystemintro}
\begin{aligned}
	\dot X &=  \sum_{i=1}^N \e^iF_i (t,X,\mu) + \e^{N+1} \tilde{F}(t,X,\mu,\e) \\
		&{\rm for}\quad  (t,X,\mu,\e) \in \R \times D \times \Sigma  \times (-\e_0,\e_0),
\end{aligned}
\end{equation}
where $D$ is an open, bounded neighbourhood of the origin in $\R^n$; $\Sigma$ is an open, bounded neighbourhood of $\mu_* \in \R^k$; $\e_0>0$;  $N \in \N^*$; and the functions $F_i$ and $\tilde{F}$ are of class $C^\infty$ in $\R \times D \times \Sigma \times (-\e_0,\e_0)$, and $T$-periodic in the variable $t$ in $\R \times \overline{D} \times \overline{\Sigma} \times[-\e_0,\e_0]$. 

We concern ourselves with $T$-periodic solutions of \cref{eq:standardsystemintro}. Let $X(t,t_0,X_0,\mu,\e)$ be the solution of this system satisfying $X(t_0,t_0,X_0,\mu,\e)=X_0$. Suppose that the parameters $(\mu,\e)$ are fixed. To find $T$-periodic solutions, we will study the so-called stroboscopic Poincaré map $\Pi$, which is defined by
\begin{equation}\label{eq:stroboscopicdefinition}
	\Pi(X_0,\mu,\e) = X(T,0,X_0,\mu,\e).
\end{equation}
Since all functions present in the system are $T$-periodic, a fixed point of the Poincaré map corresponds to a $T$-periodic solution of \cref{eq:standardsystemintro}. 

If we allow the parameters to vary, different maps emerge, giving birth to a $(k+1)$-parameter family of maps. In order to obtain a geometric picture of how the fixed points of $\Pi$ change as the parameters $(\mu,\e)$ vary, we define the catastrophe surface $M_\Pi$ of $\Pi$ as the set of triples $(X,\mu,\e)$ such that $\Pi(X,\mu,\e)=X$. 
\begin{definition} \label{definitionMpi}
	The catastrophe surface $M_\Pi$ of the Poincaré map $\Pi$ is defined by
	\begin{equation}
		M_\Pi := \{(X,\mu,\e) \in D \times \Sigma \times (-\e_0,\e_0) : \Pi(X,\mu,\e) = X\}.
	\end{equation}
\end{definition}
This definition is inspired by a similar concept appearing in Thom's catastrophe theory (see, for example, \cite{hayes1993catastrophe}). We remark that the term ``surface" is used only for reasons of custom, and does not imply that $M_\Pi$ is, globally or locally, a regular manifold in $D \times \Sigma \times (-\e_0,\e_0)$. We will see in \cref{maintheorem} below that typically $M_\Pi$ is not a manifold for the cases we will be treating. 

In this paper, we provide results locally characterising the catastrophe surface of $\Pi$ near bifurcation points for determinate classes of systems of the form \eqref{eq:standardsystemintro}. Crucially, the results show that the knowledge of an averaged form of the system - the so-called guiding system - is, in many instances, sufficient to fully describe $M_\Pi$. Essentially, we can infer in those cases that only the averaged effect of the time-dependent terms of \eqref{eq:standardsystemintro} alter the qualitative behaviour of $T$-periodic solutions.

We will need to distinguish between variables, bifurcation parameters, and perturbation parameters, and to do this we use the notion of {\it fibred maps} summarized in \Cref{sec:fibredmap}. We then give a brief introduction to the method of averaging in \Cref{sec:averagingintro}. These set up the main result in \Cref{sec:maintheorem}, showing what we define as `persistence' of catastrophes under averaging, and this is then refined to describe [non]-persistence of [non]-stable bifurcations in \Cref{sec:bifstab}-\ref{sec:bifnonstab}. Lastly, we give a comment on topological equivalence in \Cref{sec:topeq}.

\subsection{Fibred maps}\label{sec:fibredmap}

Below we will be studying how the geometry of a catastrophe surface is preserved under changes of coordinates, but we will also need to preserve the different roles variables versus bifurcation or perturbation parameters ($x,\mu,\e,$ respectively). 

Consider the simplest case $n=k=1$ (one variable and one bifurcation parameter), and compare the maps $X\mapsto \Pi_1(X,\mu,\e)=X+ X^2-\mu$ and $X \mapsto \Pi_2(x,\mu,\e) = \mu^2$. The catastrophe surfaces of those maps are, respectively, given by $M_1=\{(X,\mu,\e): X^2=\mu\}$ and $M_2:=\{(X,\mu,\e): \mu^2=X\}$. It is thus clear that $M_2$ can be obtained from $M_1$ from the rigid transformation of coordinates that rotates around the $\e$-axis by 90 degrees. Geometrically, thus, $M_1$ and $M_2$ are essentially identical. 
However, the different roles played by the coordinate $X$ and the parameters $\mu$ and $\e$, mean that $\Pi_1$ undergoes a fold bifurcation at $\mu=0$ and any $\e \in (-\e_0,\e_0)$, while $\Pi_2$ has exactly one fixed point $X^*=\mu^2$ for any pair $(\mu,\e)$. Hence, even though $M_1$ and $M_2$ are geometrically indistinguishable, the dynamics represented by them are certainly not equivalent.

This happens because the ambient space of $M_\Pi$ is the product between the space of coordinates and the space of parameters. Thus, if we want a tool that guarantees that two systems are dynamically related by comparing their catastrophe surfaces, then more than geometric properties alone, we also need the difference between parameters and coordinates to be preserved. We do this using the concept of fibred maps.

\begin{definition}
	Let $U \subset D \times \Sigma \times (-\e_0,\e_0)$ be a neighbourhood of the origin. In the context established in this paper, a map $\Phi=(\Phi_1,\Phi_2,\Phi_3): U \to \R^n \times \R^k \times \R$ is said to be:
	\begin{itemize}
		\item weakly fibred if it is of the form $\Phi(x,\mu,\e) = (\Phi_1(x,\mu,\e),\Phi_2(\mu,\e),\Phi_3(\mu,\e));$
		\item strongly fibred if it is of the form $\Phi(x,\mu,\e) = (\Phi_1(x,\mu,\e),\Phi_2(\mu,\e),\Phi_3(\e));$
		\item weakly or strongly fibred to the $m$-th order at $p \in U$ if its $m$-jet at $p$ is, respectively, weakly or strongly fibred.
	\end{itemize}
\end{definition}

For instance, when considering $M_1$ and $M_2$ as above, it is clear that, even though those surfaces are geometrically identical, no fibred diffeomorphism exists taking one into the other. 
Compare this to the map $X \mapsto \Pi_3(X,\mu,\e) = (X-\mu)^2 + X - \mu$, which still undergoes a fold bifurcation at $\mu=0$, but has a catastrophe surface $M_3= \{(X,\mu,\e) : (X-\mu)^2=\mu\}$, which can be obtained by transforming $M_1$ via the fibred diffeomorphism $\Phi(X,\mu,\e) = (X+\mu,\mu,\e)$.
This can be seen as a consequence of the assertion that fibred diffeomorphisms, by ensuring the separation of coordinates and parameters, preserve the dynamical aspects of the catastrophe surface. This admittedly somewhat vague assertion is simply a more explicit statement of ideas already present in the literature (see, for instance, the definitions of topological equivalence of families in \cite{Arnol'd1994,Kuznetsov2023}).

While weak fibration is sufficient to ensure the proper separation of coordinates and parameters, strong fibration is needed for applications such as averaging, when we need to distinguish between the perturbation parameter $\e$ and other parameters, in particular here we will usually assume $\e$ to be fixed, while studying the bifurcation family arising from varying $\mu$. 

Slightly altering the example we have already discussed, we consider the map $ X \mapsto \Pi_4(X,\mu,\e) = (X-\mu+\e)^2 +X-\mu+\e$, which, for each fixed $\e$, still has a fold as we vary $\mu$. Its catastrophe surface $M_4=\{(X,\mu,\e) : (X-\mu+\e)^2=\mu-\e\}$ is equal to the image of $M_1$ via the strongly fibred diffeomorphism $\Phi(X,\mu,\e) = (X+\mu,\mu+\e,\e)$.

It is easy to verify that the composition of two fibred maps is still fibred. It also holds that the inverse of a fibred diffeomorphism is itself fibred. These observations culminate in the following important result concerning the germs of fibred local diffeomorphisms, the proof of which can be found in \Cref{sec:fibred}.
\begin{proposition}\label{propositionfibredgroup}
	The class of germs of weakly fibred local diffeomorphisms near the origin is a group with respect to composition of germs, and so is the class of germs of strongly fibred local diffeomorphisms near the origin.  
\end{proposition}
The concept of germs of local diffeomorphisms is introduced with more detail in \Cref{sec:germsandkequivalence}.

\subsection{Averaging method and guiding system} \label{sec:averagingintro}

The averaging method allows us to simplify \cref{eq:standardsystemintro} by transforming it into a system that does not depend on time up to the $N$-th order of $\e$. More precisely, we are supplied with a smooth $T$-periodic change of variables $X \rightarrow x(t,X,\mu,\e)$ transforming \cref{eq:standardsystemintro} into
\begin{equation} \label{eq:completeaveragedsystemintro}
	\dot x = \sum_{i=1}^N \e^i g_i(x,\mu) + \e^{N+1} \, r_N (t,x,\mu,\e),
\end{equation}
where $r_N$ is $T$-periodic in $T$ and each of the functions on the right-hand side are smooth. The periodicity of this change of variables allows us to conclude that $T$-periodic solutions of \cref{eq:completeaveragedsystemintro} correspond one-to-one with $T$-periodic solutions of \cref{eq:standardsystemintro}. 

Further details about the transformation taking \cref{eq:standardsystemintro} into \cref{eq:completeaveragedsystemintro} will be provided in \Cref{sec:averagingpresentation} (in particular, \cref{lemmaaveragingtransform}), here we give just a brief overview of which elements of \cref{eq:completeaveragedsystemintro} will be used to deduce general properties of the catastrophe surface.

We obtain $g_1$ by
\begin{equation}
	g_1(x,\mu) = \frac{1}{T} \int_0^T F_1(t,x,\mu) dt,
\end{equation}
the average of $F_1$ over $t \in [0,T]$. If $g_1$ does not vanish identically, then $\ell =1$ and we are done. However, if $g_1 =0$, we proceed similarly, defining $g_2$ to be the average of an expression involving the functions $F_1$ and $F_2$ over $t \in[0,T]$. Once again, we have to check whether $g_2=0$. If not, $\ell=2$ and we are done, otherwise we have to continue in the same fashion. We do so until we reach the first $g_\ell$ that does not vanish identically. The expressions used to calculate the functions $g_i$ and other details about the transformation of variables taking \cref{eq:standardsystemintro} into \cref{eq:completeaveragedsystemintro} will be provided in \Cref{sec:averagingpresentation}.

Note that the change of variables provided by the averaging method is the identity for $t=0$, so that $M_\Pi$ can be identified with the catastrophe surface of the stroboscopic Poincaré map of \cref{eq:completeaveragedsystemintro}. Henceforth, we will always take into account this identification, since as a rule we will be analysing \cref{eq:completeaveragedsystemintro} instead of \cref{eq:standardsystemintro} directly.

Assume that at least one of the elements of $\{g_1,\ldots, g_{N-1}\}$ is non-zero and let $\ell \in \{1,\ldots, N-1\}$ be the first positive integer for which $g_\ell$ does not vanish identically. Then, \cref{eq:completeaveragedsystemintro} can be rewritten as
\begin{equation} \label{eq:completeellaveragedsystemintro}
	\dot x = \e^\ell g_\ell(x,\mu) + \e^{\ell+1} \, R_\ell (t,x,\mu,\e),
\end{equation}
where 
\begin{equation}\label{eq:definitionRell}
	R_\ell(t,x,\mu,\e) = \sum_{j=0}^{N-\ell-1} \e^j g_{j+\ell+1}(x,\mu) + \e^{N+1}r_N(t,x,\mu,\e).
\end{equation} 
The system $\dot x = g_\ell(x,\mu)$, obtained by truncating \cref{eq:completeellaveragedsystemintro} at the $\ell$-th order of $\e$ and rescaling time, is called the \textit{guiding system} of \cref{eq:standardsystemintro}. Our aim will be to infer properties of $M_\Pi$ from the singularity type appearing in the guiding system.

\subsection{Statement of the main result} \label{sec:maintheorem}

A celebrated result of the averaging method is that, if the guiding system has a simple equilibrium, then \cref{eq:standardsystemintro} has a $T$-periodic orbit for small $\e$ (see \cite{Sanders2007}). Complementarily, our main result concerns the case when
\begin{equation}\label{gguide}
	\dot x = g_\ell (x,\mu)
\end{equation}
has a singular equilibrium point at the origin for $\mu$ equal to a critical value $\mu_*$. Without loss of generality, we assume that $\mu_*=0$, that is,
\begin{enumerate}[label=(H\arabic*)]
	\item \label{hypothesissingular1} $g_\ell(0,0) =0$;
	\item \label{hypothesissingular2} $\det \left(\frac{\partial g_\ell}{\partial x}(0,0)\right)  = 0$.
\end{enumerate}

In that case, we can state the following general result, which assumes that the family $g_\ell(x,\mu)$ containing the singular equilibrium is $\mathcal{K}$-universal, that is, ``stable" in the sense of contact or $\mathcal{K}$-equivalence. More details about this concept, which is very useful in singularity theory (see, for instance, \cite{Golubitsky1985,Martinet76,Montaldi_2021}), will be given in \Cref{sec:germsandkequivalence}
\begin{theorem} \label{maintheorem}
	Let $\dot x =g_\ell(x,\mu)$ be the guiding system associated with \cref{eq:standardsystemintro}, and assume that the vector field $x \mapsto g_\ell(x,0)$ has a singular equilibrium at the origin, i.e., \cref{hypothesissingular1,hypothesissingular2} hold.  If the germ of $x \mapsto g_\ell(x,0)$ at $x=0$ has finite codimension and the $k$-parameter family $(x,\mu) \mapsto g_\ell(x,\mu)$ is a $\mathcal{K}$-universal unfolding of this germ, then there are neighbourhoods $U,V \subset \R^{n+k+1}$ of the origin and a strongly fibred diffeomorphism $\Phi: U \to V$ such that $\Phi(x,0,0)=(x,0,0)$, and the catastrophe surface $M_\Pi$ of the family of Poincaré maps $\Pi(x,\mu,\e)$ satisfies
	\begin{equation}
		M_{\Pi} \cap V = \Phi \left((Z_{g_\ell} \times \R) \cap U \right) \cup V_{\e=0},
	\end{equation}
	where $Z_{g_\ell}=\{(x,\mu) \in \R^{n+k} : g_\ell(x,\mu)  = 0\}$ and $V_{\e=0}:=\{(X,\mu,0) \in V\}$. Additionally, the set $Z_{g_\ell} \times \{0\}$ is invariant under $\Phi$. 
\end{theorem}
Observe that the set $Z_{g_\ell} \times \R$ appearing in the theorem is the catastrophe surface of the Poincaré map of the extended guiding system $\dot x = g_\ell(x,\mu), \,\dot t=1$. Hence, the theorem says that $M_\Pi$ consists in the union of two sets: a trivial part corresponding to $\e=0$, since every point is a fixed point of $\Pi$ in that case; and a non-trivial part that is, near the origin, the image under a strongly fibred diffeomorphism of the catastrophe surface of the extended guiding system. 
An illustration is given in \cref{fig:foldMPi} for a fold catastrophe. 

\begin{figure}[h!]\centering
	\includegraphics[width=0.65\textwidth]{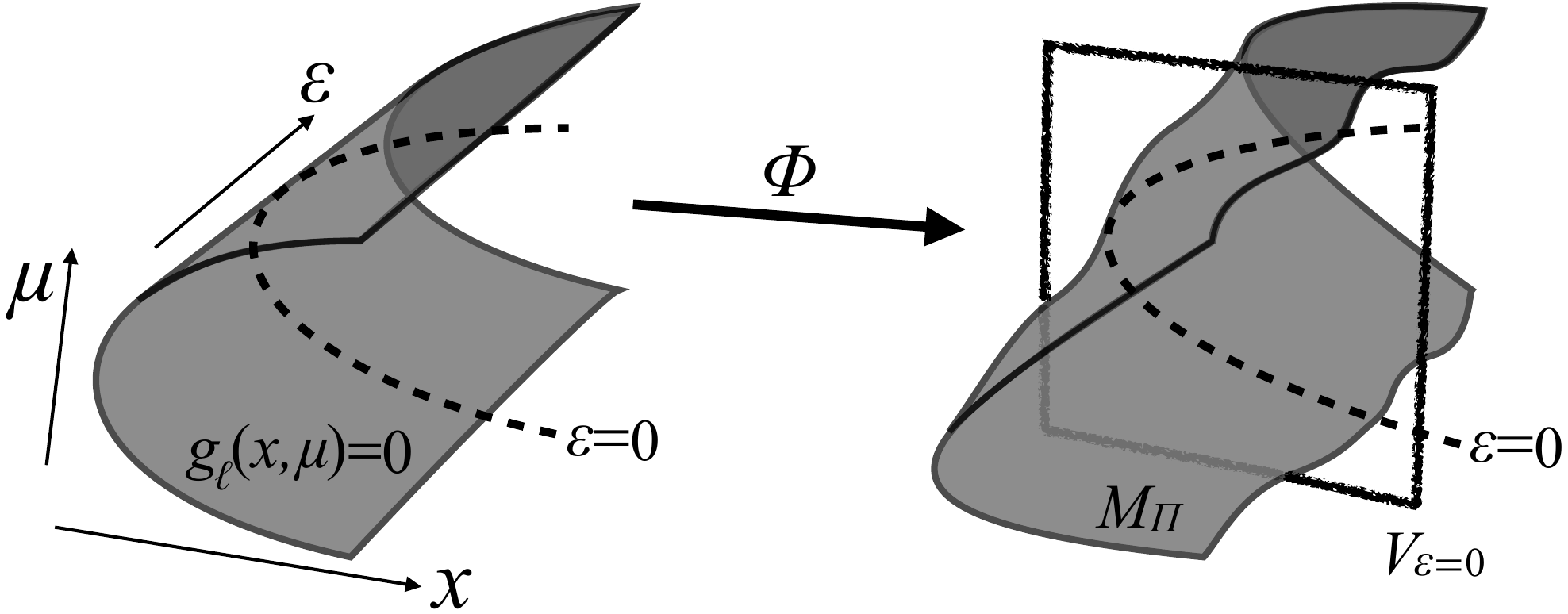}\vspace{-0.2cm}
	\caption{\sf The catastrophe surface $Z_{g_\ell}$ of the guiding system (left, suspended through $\e\in\mathbb R$), and the catastrophe surface $M_\Pi$ of the time-dependent system (right). $M_\Pi$ is the image of $Z_{g_\ell}$ under the diffeomorphism $\Phi$, and $Z_{g_\ell}\times\{0\}$ is invariant under $\Phi$. 
	}\label{fig:foldMPi}
\end{figure}
%

\subsection{Persistence of bifurcation diagrams for stable families}\label{sec:bifstab}

A specially illustrative way to look at \cref{maintheorem} is as ensuring `persistence' of the well-known bifurcation diagrams of fixed points for $\mathcal{K}$-universal (also known as stable) families.

For a general family of vector-fields $\dot x = F(x,\eta)$, the bifurcation diagram of equilibria is the subset of the coordinate-parameter space defined by $\{(x,\eta): F(x,\eta)=0\}$. Analogously, for a general family of maps $(x,\eta) \mapsto P(x,\eta)$, the bifurcation diagram of fixed points is $\{(x,\eta): P(x,\eta)=x\}$.

In the averaging method, the guiding system can be seen as the first non-trivial approximation of a system. It is thus desirable to determine to which degree this approximation allows us to extrapolate qualitative properties to the original system. 

In the case treated in this paper, the guiding system is actually a family of vector fields undergoing some local bifurcation. The original system \cref{eq:standardsystemintro}, however, has one extra perturbative parameter $\e$ and is non-autonomous, so that the manner of comparison of its qualitative properties with those of the guiding system is not obvious. 

To make this comparison possible, we fix $\e \neq 0$ small and compare the bifurcation diagrams of $\dot x = g_\ell(x,\mu)$ and $(x,\mu) \mapsto \Pi(x,\mu,\e)$, that is, we see the parameter $\e$ as a perturbation of the bifurcation diagram of the guiding system. We can then reinterpret \cref{maintheorem} as stating that, for $\mathcal{K}$-universal families, the bifurcation diagrams of fixed points of the perturbed maps are actually $\mathcal{O}(\e)$ perturbations of the bifurcation diagram of equilibria of the guiding system, as follows.

\begin{theorem}\label{theorembifdiagram}
	Under the hypotheses of \cref{maintheorem}, the bifurcation diagram $\mathcal{D}_{\ell,0} : =\{(x,\mu) \in D \times \Sigma: g_\ell(x,\mu)=0\}$ is locally a smooth manifold of codimension $k$ near the origin. For $\e\neq0$ sufficiently small, the perturbed bifurcation diagrams $\mathcal{D}_\e:=\{(x,\mu) \in D \times \Sigma: \Pi(x,\mu,\e)=x\}$ are also smooth manifolds of codimension $k$ near the origin, which are $\mathcal{O}(\e)$-close to $\mathcal{D}_{\ell,0}$.
\end{theorem}

This will be proven in \Cref{sec:proofbif}, and can essentially be stated as \textit{persistence} of the bifurcation diagram from the guiding system \cref{gguide} to the full $\e$-perturbed system \cref{eq:standardsystemintro} for small values of $\e$. The statement of the result in terms of persistence of qualitative properties of the guiding system is intended to mirror a selection of results in the area (see \cite[Chapter 6]{Sanders2007} and, more recently, \cite{CANDIDO20204555,PEREIRA20231}.

\subsection{Stabilisation of non-stable families}\label{sec:bifnonstab}

For non-stable families the bifurcation diagrams will not typically persist, as we show below for the transcritical and pitchfork bifurcations, which instead form a pair of folds and a cusp, respectively. The analysis of those two families is only meant to illustrate how the method can still be applied (not directly and with due caution) if a non-versal bifurcation appears in the guiding system, by `embedding' this bifurcation into a larger versal family. 

The specific choice of the transcritical and the pitchfork is motivated by the fact that they are one-parameter families that appear often in the literature, due to their natural connection with symmetries or other constraints of the system.  For instance, the pitchfork bifurcation is versal in the context of germs having $\mathbb{Z}_2$-symmetry, whereas the transcritical is versal if germs are required to have $0$ as an equilibrium point (for more details, see \cite[Chapter 23]{Montaldi_2021}). The stabilisation process explained in this section can thus be seen as the effect time-periodic perturbations have in breaking symmetries - or, more generally, removing constraints - of the model.

\subsubsection{Transcritical}
\begin{theorem} \label{theoremtranscriticalcatastrophe}
	Let $n=1$ and suppose that the guiding system $\dot x =g_\ell(x,\mu)$ undergoes a transcritical bifurcation at the origin for $\mu=0$. If
	\begin{equation}
		g_{\ell+1}(0,0) \neq 0,
	\end{equation} 
	then there are neighbourhoods $U,V \subset \R^{1+1+1}$ of the origin and a strongly fibred diffeomorphism $\Phi: U \to V$ such that the catastrophe surface $M_\Pi$ of the family of Poincaré maps $\Pi(x,\mu,\e)$ satisfies
	\begin{equation}
		M_\Pi \cap V = \Phi \left(\{(y,\theta,\eta) \in \R \times \R \times \R : \eta= y^2- \theta^2\} \cap U \right) \cup V_{\e=0},
	\end{equation}
	where $V_{\e=0}:=\{(X,\mu,0) \in U \}$. Additionally, $\Phi(0,0,0)=(0,0,0)$,
	\begin{equation}
		\sgn\left(\Phi'_3(0)\right) = \sgn\left(\frac{\partial^2 g_\ell}{\partial x^2}(0,0)\right) \cdot \sgn \left(g_{\ell+1}(0,0)\right),
	\end{equation} and
	\begin{equation}
		(Z_{g_\ell} \times \{0\}) \cap V = \Phi\left(\{(y,\theta,0) \in \R \times \R \times \R: y^2-\theta^2=0\} \cap U\right).
	\end{equation}
\end{theorem}

\subsubsection{Pitchfork}
\begin{theorem}\label{theorempitchforkcatastrophe}
	Let $n=1$ and suppose that the guiding system $\dot x =g_\ell(x,\mu)$ undergoes a pitchfork bifurcation at the origin for $\mu=0$. If
	\begin{equation}
		g_{\ell+1}(0,0) \neq 0,
	\end{equation} 
	then there are neighbourhoods $U,V \subset \R^{n+1+1}$ of the origin and a weakly fibred diffeomorphism $\Phi: U \to V$ such that the catastrophe surface $M_\Pi$ of the family of Poincaré maps $\Pi(x,\mu,\e)$ satisfies
	\begin{equation}
		M_\Pi \cap V = \Phi \left(\{(y,\theta,\eta) \in \R \times \R \times \R : y^3+ y \theta + \eta=0 \} \cap U \right) \cup V_{\e=0},
	\end{equation}
	where $V_{\e=0}:=\{(X,\mu,0) \in U \}$. Additionally, $\Phi(0,0,0)=(0,0,0)$, and $\Phi$ is strongly fibred to the first order at the origin.
\end{theorem}

We will illustrate these results with examples in \Cref{sec:examples}.

\subsection{A discussion of topological equivalence} \label{sec:topeq}

In studying bifurcations of dynamical systems, it is common to work with topological equivalence classes. As noted in \cite{j22cat}, this is practically restrictive, and we will instead work only with the bifurcations of numbers of equilibria, better termed {\it catastrophes} as they ignore topological properties of the dynamics, also referred to {\it underlying catastrophes} in \cite{j22cat}. 

Generally speaking, the catastrophe surface alone does not determine the topological class of the Poincaré map: there are potentially multiple topological classes with the same catastrophe surface. However, knowing the catastrophe surface reduces the number of possibilities for the topological types of the map, and we can see it as one of the elements constituting a general topological description.

Let us briefly explore this distinction by presenting the saddle-node case in one dimension, for which the catastrophe surface allows us to very easily infer topological conjugacy, and also exhibiting an interesting counter-example for planar vector fields.

\subsubsection{The saddle-node in one-dimension} \label{sec:onedimensionaltopologicalequivalence}

In the case of well-studied one-dimensional stable bifurcations, we can assert the topological conjugacy class of $\Pi_\e$ by combining the method exposed in this paper with known genericity conditions ensuring topological conjugacy to the normal form for the bifurcation (see \cite[Theorems 4.1 and 4.2]{Kuznetsov2023}).

\begin{theorem} \label{theoremsaddlenode}
	Let $n=1$ and suppose that the guiding system $\dot x = g_\ell (x,\mu)$ undergoes a saddle-node bifurcation at $(0,0)$, that is, assume that the following conditions hold:
	\begin{enumerate}[label=(F\arabic*)]
		\item \label{saddlenode1}$\frac{\partial g_\ell}{\partial \mu}(0,0)\neq 0$;
		\item \label{saddlenode2}$\frac{\partial^2 g_\ell}{\partial x^2}(0,0) \neq 0 $.
	\end{enumerate}
	Then, there are $\e_1 \in (0,\e_0)$, and smooth functions $x^*:(-\e_1,\e_1) \to D$ and $\mu^*:(-\e_1,\e_1) \to \Sigma$ such that:
	\begin{enumerate}[label=(\roman*)]
		\item $(x^*(0),\mu^*(0))=(0,0)$.
		\item For each $\e \in (-\e_1,\e_1) \setminus\{0\}$ fixed, the family of stroboscopic Poincaré maps $(x,\mu)\mapsto \Pi(x,\mu,\e)$ is locally topologically conjugate near $(x^*(\e),\mu^*(\e))$ to one of two possible normal forms: $(y,\beta)  \mapsto (\beta-\mu^*(\e)) +(y-x^*(\e)) \pm (y-x^*(\e))^2$. In other words, the family of $(x,\mu) \mapsto \Pi(x,\mu,\e)$ is, up to translation of coordinates, locally topologically conjugate to one of the two topological normal forms for the saddle-node bifurcation for maps: $(y,\beta) \mapsto \beta + y \pm y^2$.
	\end{enumerate}  
\end{theorem}

The proof of this theorem is located in \cref{sec:proofs}. An analogous result can be obtained for the cusp bifurcation, by considering the conditions available in \cite[Theorem 9.1]{Kuznetsov2023}.

\subsubsection{The saddle-focus in two-dimensions}  \label{sec:saddlefocus}

We looked at systems with \textit{saddle-node} bifurcations in \Cref{sec:onedimensionaltopologicalequivalence}. The saddle-node is well known to be a generic one parameter bifurcation under topological equivalence, with normal form given by $(\dot x_1,\dot x_2,\ldots,\dot x_n)=\left(x_1^2+\mu,\;x_2,\ldots, x_n\right)$.

However, the collision between a saddle and a focus, which we will refer to as a saddle-focus, is not a generic one parameter bifurcation, but an example one parameter family is obtained if we interchange two entries on the right-hand side of the normal form family of the saddle-node: $(\dot x_1,\dot x_2,\ldots, \dot x_n)=\left(x_2,\;x_1^2+\mu,\ldots, x_n\right)$.

A generic family with a saddle-focus is found, for example, in the well studied Bogdanov-Takens bifurcation (see \cite{Arnol'd1994}, \cite{Guckenheimer1983}, or \cite{Kuznetsov2023})
\begin{equation}
	(\dot x_1,\dot x_2)=\left(x_2,\;x_1^2\pm x_1 x_2 +\mu_2 x_1^2 +\mu_1 \right)\;,
\end{equation}
requiring not a single parameter $\mu_1$ to unfold it, but also $\mu_2$ to control the local appearance of limit cycles and homoclinic connections (hence we have $\mu=(\mu_1,\mu_2)$). The singular germ corresponding to this family, obtained for zero values of the parameters, is $\left(\dot x_1 , \dot x_2\right) = (x_2, x_1^2 + x_1 x_2)$, and the two parameters appearing in the Bogdanov-Takens bifurcation ensure that this germ is of codimension 2 when considering topological equivalence.

However, it is interesting to notice that, if regarded as the germ of a plane map, the germ of $(x_1,x_2) \mapsto (x_2,x_1^2+x_1 x_2)$ is actually $\mathcal{K}$-equivalent to $(x_1,x_2) \mapsto (x_2,x_1^2)$, which is itself equivalent to the saddle-node germ $(x_1,x_2) \mapsto (x_1^2,x_2)$. Essentially, this means that, with respect to the unfolding of zeroes of those map germs, i.e., equilibria of the corresponding vector fields, all three germs behave identically. Naturally, this observation does not allow us to obtain a complete description of the phase portrait of a family, as they are topologically different, but certain properties -- namely the numbers of equilibria and hence the catastrophe surface -- can still be fully understood.

In particular, we can describe the unfolding of equilibria of the germ of the vector field $(\dot x_1,\dot x_2) = (x_2, x_1^2)$, for which a complete unfolding with respect to topological equivalence is not known, and probably not even possible, hence the alternative germs unfolded in the Bogdanov-Takens bifurcation \cite{Bogdanov1976,Takens1974}, versus the Dumortier-Roussarie-Sotomayor bifurcation \cite{Dumortier1987}. 

The latter of these provides a different generic family with a saddle-focus configuration,
\begin{equation}\label{DRS}
	(\dot x_1,\dot x_2)=\left(x_2,\;x_1^2+\mu_1+x_2(\mu_2+\mu_3x_1+x_1^3)\right)\;.
\end{equation}
Like the Bogdanov-Takens bifurcation, this family requires not just the parameter $\mu_1$ to unfold it, but in this case two other parameters, $\mu_2$ and $\mu_3$. The singular germ of this family is $\left(\dot x_1 , \dot x_2\right) = (x_2, x_1^2 + x_1^3 x_2)$. We will use this example to illustrate persistence of the catastrophe, irrespective of topological equivalence, in \Cref{sec:counterexample}.


\section{The theory in practice: time-periodic coefficients} \label{sec:examples}

Before setting out the theory from \Cref{sec:results} in detail, let us show how it works in practice on a few examples. For these we take the interesting applied problem of a system whose parameters are not exactly fixed, but vary slightly over time. To apply averaging we will assume that variation is periodic, for instance a physiological model in which some hormones have a small circadian perturbation, or a climate model where temperature has a small daily fluctuation.

The examples we treat here are intentionally simple, thus could be studied with other methods not relying on the averaging method - in particular, time-periodicity is not essential. However, they are meant only to \textit{illustrate} the general results obtained in this paper, hence the choice for simple settings. 

\subsection{Example: persistence of fold catastrophe} \label{sec:egfold}

Consider a system $\dot Y=Y^2$, perturbed by a parameter of order $\e^2$ and with a $T$-periodic fluctuation, 
\begin{equation}
	\dot Y = Y^2+\e^2 f(t)\;.
\end{equation}
For $\e\neq0$, the change of variables $Y= \e X$ transforms this into
\begin{equation} \label{eq:examplefold1}
	\dot X = \e \left(X^2 + f(t)\right),
\end{equation}
which is a family of systems in the standard form \cref{eq:standardsystemintro}. 

If we define the average of $f(t)$ as
\begin{equation}
	\mu : = \frac{1}{T} \int_0^T f(t)dt
\end{equation}
and the oscillating part of $f$ to be $\tilde{f}(t) = f(t) - \mu$, we have
\begin{equation}\label{eq:examplefold2}
	\dot X = \e \left(X^2 + \mu + \tilde{f}(t)\right),
\end{equation}
where the average of $\tilde{f}$ over $[0,T]$ is zero. Accordingly, the stroboscopic Poincaré map of \cref{eq:examplefold2} will be denoted by $\Pi$ and its catastrophe surface by $M_\Pi$.

Applying the transformation of variables given by the averaging theorem to obtain a system of the form \cref{eq:completeaveragedsystemintro}, this family becomes
\begin{equation}\label{eq:examplefold2av}
	\dot x = \e (x^2+\mu) + \e^2 R_1 (t,x,\mu,\e).
\end{equation}
It is then clear that the guiding system $\dot x = x^2+\mu$ undergoes a fold bifurcation for $\mu=0$, which corresponds to a $\mathcal{K}$-universal unfolding of the singular germ $x^2$. To illustrate the verification of $\mathcal{K}$-universality of an unfolding, we apply the criterion presented in \cref{propositionuniversaliff}. There is only one parameter in the unfolding $F(x,\mu) = x^2 + \mu$ of $f(x) = x^2$, and the set of germs 
\begin{equation*}
	\mathcal{D}:=\left\{\left[\frac{\p F}{\p \mu}\Big|_{\mu=0} \right]\right\} = \left\{[1]\right\}
\end{equation*}
is clearly linearly independent in $\mathcal{E}_1$. Moreover, the extended $\mathcal{K}$-tangent space of $[f]$ is 
\begin{equation*}
	T_{\mathcal{K},\e} f = \left\{[x] \cdot [X] + [M] \cdot [x^2]: [X] \in \bm{X}^0_1, [M] \in \bm{M}^0_1 \right\} = \left\{[x] \cdot [X]: [X] \in \mathcal{E}_1 \right\}.
\end{equation*}
Hence, it follows from Hadamard's Lemma (for a statement, see \cite[Section 3.2]{Montaldi_2021}) that 
\begin{equation*}
	T_{\mathcal{K},\e} f \oplus \mathcal{D} = \mathcal{E}_1,
\end{equation*}
so that, by \cref{propositionuniversaliff}, it follows that $[F]$ corresponds to a $\mathcal{K}$-universal unfolding of the singular germ $[f]$.

Therefore, \Cref{maintheorem} ensures that $M_\Pi$ locally has the form of a fold surface near the origin. Consequently, for each small fixed $\e\neq0$, a fold-like emergence (or collision) of fixed points of $x\mapsto \Pi(x,\mu,\e)$ occurs near $0$ as $\mu$ traverses a neighbourhood of zero. The value of $\mu$ for which this occurs is given by a continuous function $\mu^*(\e)$ satisfying $\mu^*(0)=0$.

As an example, let $f(t)=\mu+\sin(t)$, so
\begin{equation}\label{foldeg1}
	\begin{aligned}
		\dot X&=\e(X^2+\mu+\sin(t)).
	\end{aligned}
\end{equation}
While we cannot solve this exactly, it is instructive to look at its perturbative solution for small $\e$, which is $X(t,X_0,\mu,\e)\sim x_g(t,X_0,\mu,\e)+\e(1-\cos t)+2X_0\e^2(t-\sin t)+\mathcal{O}(\e^3)$, where $x_g$ is the solution of the non-oscillatory problem $\dot x=\e(x^2+\mu)$. Averaging this system amounts to removing the order $\e$ oscillatory term by making a change of variables $X=x-\e \cos t$. 
In the method set out in \Cref{sec:averagingintro}, we have $\tilde f(t)=\sin(t)$ and $R_1(t,x,\mu,\e)=-2x\cos(t)$, giving the averaged system 
\begin{equation}\label{foldeg1av}
	\begin{aligned}
		\dot x=\e(x^2+\mu)-2\e^2x\cos(t),
	\end{aligned}
\end{equation}
\begin{figure}[h!!]\centering
	\includegraphics[width=0.7\textwidth]{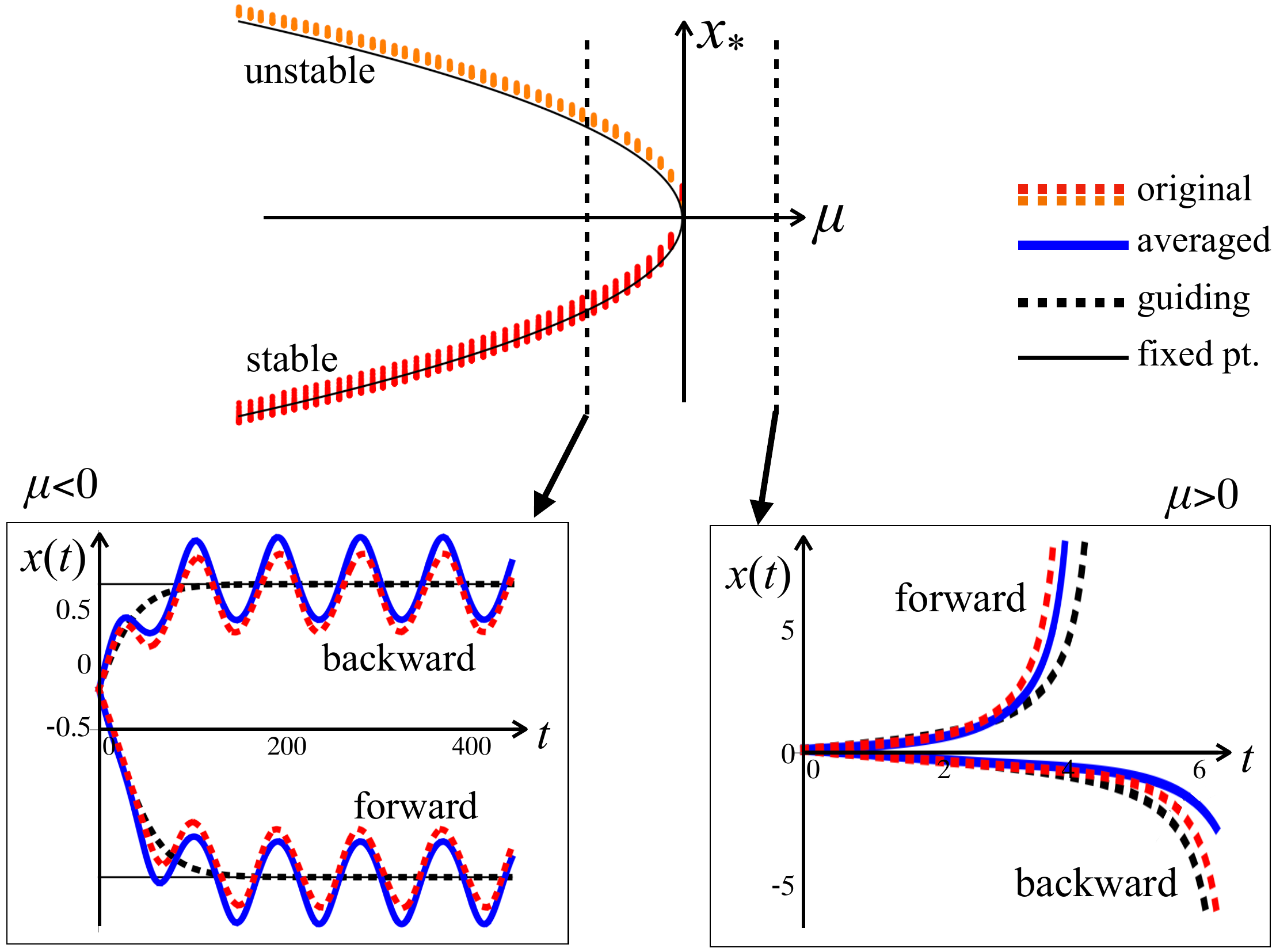}\vspace{-0.2cm}
	\caption{\sf Solutions of the system \cref{foldeg1} exhibiting a fold, with $\mu=\pm0.5$ and $\e=0.4$.
		The upper picture shows the Poincar\'e map of the original system (red/orange points), converging in forward/backward time to the fixed points of the guiding system (black curves). For two values of $\mu$ we plot the solutions below. For $\mu<\mu^*(0.5)$, the exact solutions (red dotted curves) and averaged solutions (blue curves) oscillate around the guiding solutions (black dotted curves), all converging in forward/backward time onto the stable/unstable fixed points (black lines). For $\mu>\mu^*(0.5)$ there are no fixed points and the solutions diverge. (For the exact solution we plot the variable $x=X+\e\cos t$). 
	}\label{fig:foldsol}
\end{figure}
whose solutions satisfy $x(t,x_0,\mu,\e)\sim x_g(t,x_0,\mu,\e)-2x_0\e^2\sin t+\mathcal{O}(\e^3)$, where the oscillation has moved to higher order. 
The guiding system $\dot x = \e (x^2+\mu)$ can be solved exactly, and its solutions are
\begin{equation}\label{foldeg1guide}
	\begin{aligned}
		x_g(t,x_0,\mu,\e)&=\sqrt{-\mu} \tanh\left(\e\sqrt{-\mu} t+\arctanh \left(\frac{x_0}{\sqrt{-\mu}}\right)\right)\\
		&\sim x_0+(x_0^2+\mu)\e t+(x_0^2+\mu)x_0\e^2t^2+\mathcal{O}(\e^3).
	\end{aligned}
\end{equation}
These different solutions are illustrated in \cref{fig:foldsol}, and we see the consequence of the results proven above, that for $\mu<0$ the solutions of the exact and averaged systems all tend towards oscillation around the fixed points of the guiding system, but as $\mu$ moves to positive values a fold occurs and the fixed points vanish.

\subsection{Example: non-persistence of the transcritical bifurcation}\label{sec:egtrans}

Consider the following differential system, with two time-dependent coefficients at different orders of $\e$,
\begin{equation}
	\dot Y = Y^2+ \e f_1(t) Y + \e^2 f_2(t).
\end{equation}
We assume $f_1$ and $f_2$ to be $T$-periodic. 
The change of variables $Y=\e X$ for $\e\neq0$ yields
\begin{equation}
	\dot X = \e \left(X^2 + f_1(t) X + \e f_2(t) \right).
\end{equation}

Define the averages of $f_1$ and $f_2$ as
\begin{equation}
	\mu := \frac{1}{T} \int_0^T f_1(t) dt,\qquad c := \frac{1}{T} \int_0^T f_2(t) dt,
\end{equation}
and the oscillating part of $f_1$ as $\tilde{f}_1(t) := f_1(t) - \mu$. The system can then be rewritten as 
\begin{equation}
	\dot X = \e X^2 + \e \mu X + \e  \tilde{f}_1(t) X + \e^2 f_2(t).
\end{equation}
This is now in the form \cref{eq:standardsystemintro} with $N=1$, $F_1(t,X,\mu) = X^2 + \mu X +  \tilde{f}_1(t) X$, and $\tilde{F}(t,X,\mu,\e)= f_2(t)$. We then apply the change of variables given by the averaging theorem, that is, $X=x-\e(x+\e\cos t)\cos t$, obtaining
\begin{equation}
\begin{aligned}
	\dot x &= \frac{\e G(t,x,\e)}{B(t,\e)}\qquad{\rm where}\\
	&G(t,x,\e) = x^2B^2(t,\e) + (\mu+ \tilde{f}_1(t))x B(t,\e) + \e f_2(t) - xA_1'(t)\;,\\
	&B(t,\e) =1+\e A_1(t)\;.
\end{aligned}
\end{equation}
and $A_1(t)$ is such that $A_1'(t) = \tilde{f}_1(t)$. Expanding in powers of $\e$, we obtain the averaged system
\begin{equation}
	\dot x = \e (x^2+\mu x ) + \e^2\left(x^2A_1(t) +x A_1(t) A_1'(t) + f_2(t) \right) + \mathcal{O}(\e^3).
\end{equation}
It follows that the guiding system is $\dot x = g_1(x,\mu) = x^2+\mu x$ and the remainder term is $R_1(t,x,\mu,\e) = x^2A_1(t) + x A_1(t) A_1'(t) + f_2(t) + \mathcal{O}(\e)$. Thus,
\begin{equation}
	g_2(0,0) = \int_0^T R_1(t,0,0,0) dt = \int_0^T f_2(t)dt=c.
\end{equation}
If $c\neq0$, the system satisfies the hypotheses of \cref{theoremtranscriticalcatastrophe}, and the perturbation causes the described stabilisation of the catastrophe surface.

For illustration, let $f_1(t)=\mu+\sin(t)$ and $f_2(t)=c+\sin(2t)$, so we are studying the system $\dot X=\e(X^2+X\mu+X\sin(t))+\e^2(c+2\sin(2t))$. Then $\tilde f_1(t)=\sin(t)$, $A_1(t)=-\cos(t)$, and $R_1(t,x,\mu,\e)=(c-(2+x)\sin(t)\cos(t)-x^2\cos(t))$. The averaging theorem yields the system $\dot x=\e(x^2+\mu x)+\e^2(c-(2+x)\sin(t)\cos(t)-x^2\cos(t))$.

Unlike stable families, the catastrophe surface of the perturbed system will not lie close to (be a diffeomorphism of) that of the guiding system's transcritical geometry. We can examine how the catastrophe surface unfolds with $\e$ by taking the second order averaging system. We plot the zeros of this as an illustration of the catastrophe surface in \cref{fig:transsurfs}, which approximates $M_\Pi-V_{\e=0}$ for small $\e$, and coincides with it at $\e=0$. Only at $\e=0$ does the transcritical appear, though strictly with null stability since the full system is $\dot x=\e x(x^2+\mu)+\mathcal{O}(\e^2)$. 
The stability of equilibria shown on the right of Figure 3 is not given by studying the catastrophe surface, but comes from simulations (or can be verified by further stability analysis).

\begin{figure}[h!]\centering
	\includegraphics[width=0.65\textwidth]{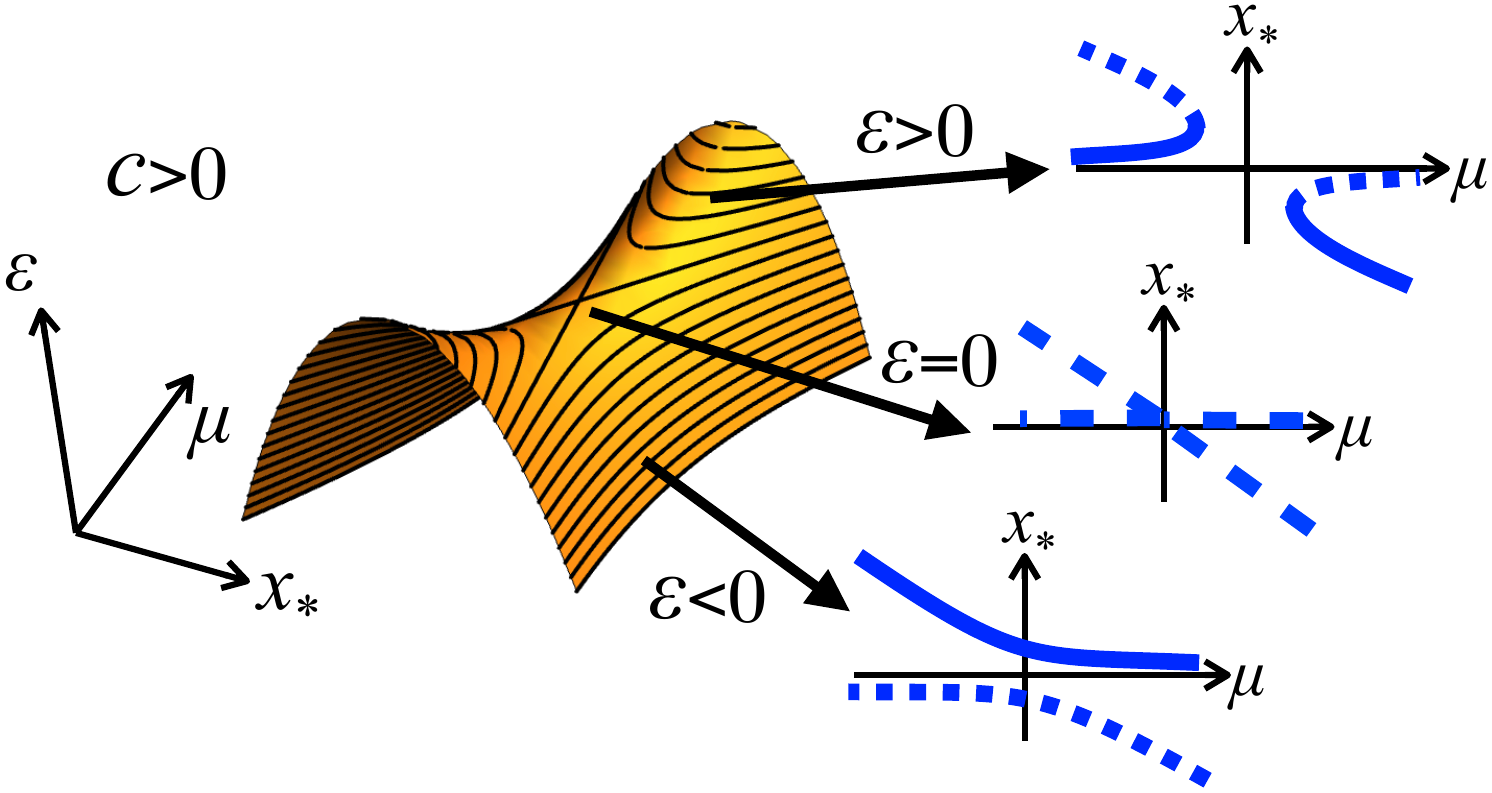}\vspace{-0.2cm}
	\caption{\sf The surface of fixed points $x^2+\mu x+\e c=0$ plotted for $c=1$ in $(x,\mu,\e)$ space (with $x_*$ denoting the fixed-point value of $x$). The transcritical bifurcation at $\e=0$ degenerates into a pair of fold bifurcations for $\e>0$ and two stable families of fixed points for $\e<0$. Sections of the surface at different $\e$ give the bifurcation diagrams with varying $\mu$ (stable/unstable branches indicated by full/dotted curves). 
	}\label{fig:transsurfs}
\end{figure}

\Cref{fig:transpoincare} shows simulations of Poincar\'e maps of the original system, which are a small perturbation of the bifurcation curves of the second order averaged system $\dot x=\e(x^2+\mu x)+\e^2c$, corresponding to the sections shown in \cref{fig:transsurfs}. 

\begin{figure}[h!]\centering
	\includegraphics[width=0.8\textwidth]{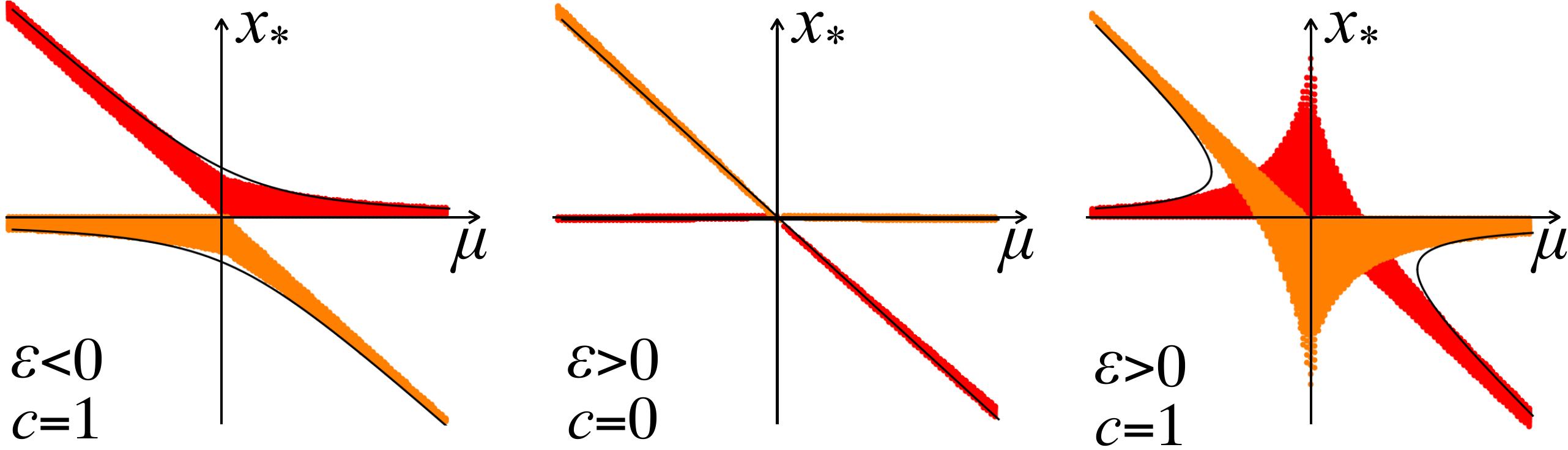}\vspace{-0.2cm}
	\caption{\sf Solutions of the perturbed transcritical system. The Poincar\'e map of $x(t)$ $\mod(t,2\pi)$, showing exact solutions converging in forward time (red) and backward time (orange) onto the stable and unstable fixed points (black curves), respectively, from initial conditions close to the fixed points if they exist, or close to the origin otherwise. The parameters used are: left $\e=-0.02,c=1$, middle $\e=0.02,c=0$, right $\e=0.02,c=1$. For $\e<0$ there are always two fixed points. For $\e>0$ and $c\neq0$ there are two fixed points only for $|\mu|>\mu_{\rm fold}$, so between the folds the solutions diverge.  
	}\label{fig:transpoincare}
\end{figure}

Lastly, \cref{fig:transsol} shows solutions for different values of $\mu$ and $\e$, showing the solutions converging onto a pair of fixed points, except for parameters values that lie between the two folds at which no fixed points exist, so solutions diverge.

\begin{figure}[h!]\centering
	\includegraphics[width=0.98\textwidth]{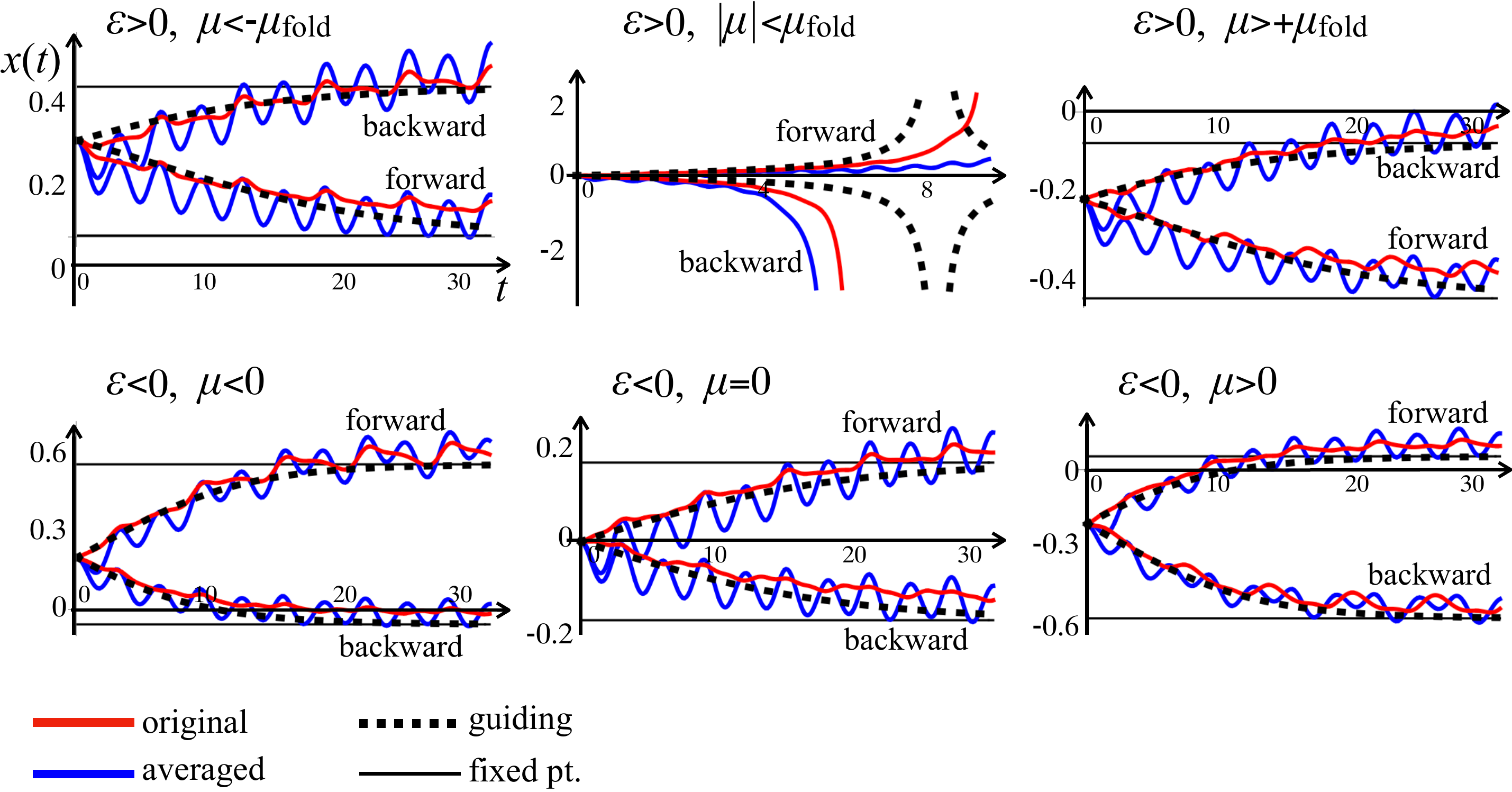}\vspace{-0.2cm}
	\caption{\sf Solutions of the perturbed transcritical system with $c=0.1$, for $\e=0.3$ (top row) and $\mu=-0.3$ (bottom row), with values $\mu=-0.5,0,0.5$, (from left to right). The original solutions (red curves) and averaged solutions (blue curves) oscillate around the guiding solutions (black dotted curves), converging in forward/backward time onto the stable/unstable fixed points (blue curves) if they exist. For $\e>0$ there are two fixed points only for $|\mu|>\mu_{\rm fold}$, so in the middle picture the solutions diverge. For $\e<0$ there are always two fixed points. 
	}\label{fig:transsol}
\end{figure}

\subsection{Example: non-persistence of the pitchfork}

Similar to the example for the transcritical, consider a system with two $T$-periodic parameters at different orders of $\e$, but this time take
\begin{equation}
	\dot Y = Y^3+ \e^2 f_1(t) Y + \e^4 f_2(t).
\end{equation}
The change of variables $Y=\e X$ for $\e\neq0$ yields
\begin{equation}
	\dot X = \e^2 \left(X^3 + f_1(t) X + \e f_2(t) \right).
\end{equation}
If $\mu$ denotes the average over $[0,T]$ of $f_1$ and $\tilde{f}_1(t):=f_1(t)-\mu$, we obtain
\begin{equation}
	\dot X = \e^2 \left(X^3 +\mu X+ \tilde{f}_1(t) X\right) + \e^3 f_2(t).
\end{equation}
This system is in the standard form with $N=2$, $F_1(t,X,\mu)=0$, $F_2(t,X,\mu) = X^2 + \mu X +  \tilde{f}_1(t) X$, and $\tilde{F}(t,x,\mu,\e)= f_2(t)$. 

We can then apply the change of variables given by the averaging theorem. Let $A_1(t)$ be such that $A_1'(t) = \tilde{f}_1(t)$. We perform the change of variables given by $X = x+ \e x A_1(t)$, obtaining
\begin{equation}
	\dot x = \e^2 (x^3+\mu x) + \e^3 \left(f_2(t) -A_1(t) \left(\mu x +x^3 \right)\right) + \mathcal{O}(\e^4).
\end{equation}
Following the naming convention of the averaging method, we have the guiding system $\dot x = g_2(x,\mu) = x^3+\mu x$ and $R_2(t,x,\mu,\e) =  f_2(t) -A_1(t) \left(\mu x +x^3 \right) + \mathcal{O}(\e)$. Thus, assuming that the average of $f_2$ over $[0,T]$ does not vanish, it follows that
\begin{equation}
	g_3(0,0) = \int_0^T R_2(\tau,0,0,0) d\tau = \int_0^T f_2(\tau) d\tau \neq 0.
\end{equation}
We are therefore within the domain of application of \cref{theorempitchforkcatastrophe}.

For illustration, let $f_1(t)=\mu+\sin(t)$ and $f_2(t)=c+\sin(2t)$, so we are studying the system $\dot X=\e^2(X^3+\mu X+ \sin(t) X)+\e^3(c+2\sin(2t))$. Then $\tilde f_1(t)=\sin(t)$, and $R_1(t,x,\mu,\e)=(c-(2+x)\sin(t)\cos(t)-2x^3\cos(t))$, hence the averaging theorem yields the system $\dot x=\e(x^3+\mu x)+\e^2(c-(2+x)\sin(t)\cos(t)-2x^3\cos(t))$. 

\Cref{fig:pitchsurfs} illustrates the cusp catastrophe surface formed by the fixed points in $(x,\mu,\e)$ space, approximating $M_\Pi-U_{\e=0}$. For $\e\neq0$ the bifurcation diagram, instead of a pitchfork, exhibits a branch of persistent equilibria and a fold. Only at $\e=0$ does the pitchfork appear, though strictly with null stability since the full system is $\dot x=\e(x^3+\mu x)+\mathcal{O}(\e^2)$.

\begin{figure}[h!]\centering
	\includegraphics[width=0.7\textwidth]{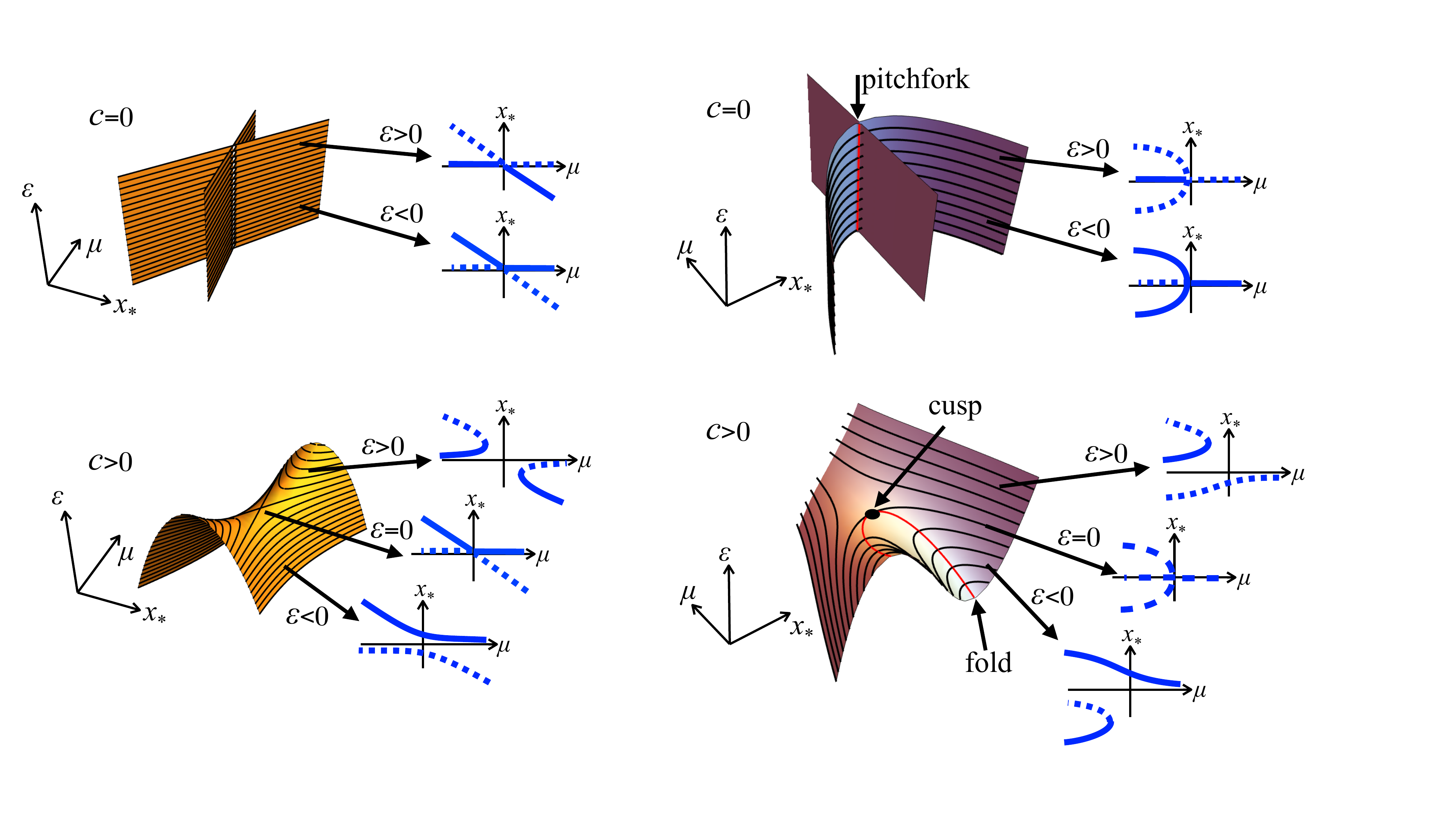}\vspace{-0.2cm}
	\caption{\sf A cusp bifurcation appearing for $c=1$. The fixed points are plotted in $(x,\mu,\e)$ space (with $x_*$ denoting the fixed point value of $x$). Sections of this at different $\e$ give the bifurcation diagrams with varying $\mu$, showing a fold and a persistent fixed point for $\e\neq0$ (stable/unstable branches indicated by full/dotted curves). 
	}\label{fig:pitchsurfs}
\end{figure}

\subsection{A counter-example to topological equivalence: the saddle-focus} \label{sec:counterexample}

Let us consider the two-dimensional family
\begin{equation}
	(\dot x_1,\dot x_2)=\e\left(x_2,\;x_1^2+\mu+\sin(t)\right)+\e^2\left(0,x_2(c_0+c_1x_1+x_1^3)\right)\;.
\end{equation}
If we omit the time-dependent term $\e\sin(t)$, this takes the form of the saddle-focus bifurcation of Dumortier-Roussarie-Sotomayor, given in \Cref{sec:saddlefocus}. 
For the averaged system we obtain
\begin{equation}
	(\dot x_1,\dot x_2)=\e\left(x_2,\;\mu+x_1^2\right)+\e^2\left(-\cos(t),x_2(c_0+c_1x_1+x_1^3)\right)\; + \mathcal{O}(\e^3).
\end{equation}
The guiding system $(\dot x_1,\dot x_2)=\e\left(x_2,\;\mu+x_1^2\right)$ is structurally unstable.

The purpose of this example is to show that we do not need structural stability in the guiding system, or topological equivalence to the form of bifurcation we derive, to apply our main theorem. Of course, we should not expect a full topological description of the map, as discussed in \Cref{sec:onedimensionaltopologicalequivalence}.

\Cref{fig:sadfocus} illustrates solutions of the exact $\e$-perturbed system, and the averaged system, for $\mu=-0.2$, for which the guiding system has a saddle and a focus (for $\mu>0$ the guiding system has no equilibria, and correspondingly all solutions of the averaged and perturbed systems diverge). The bottom row shows system where the guiding system has, for different $c_1$ and $c_0$ values from left to right: a stable equilibrium, a centre, a stable limit cycle, respectively, corresponding in the full system to a stable periodic orbit, a family of invariant tori (whether these are hyperbolic when the guiding system is perturbed would require more in-depth analysis), and a stable invariant torus. 

A magnification of the first case from \cref{fig:sadfocus} is shown in \cref{fig:sadfocuszoom} to more closely reveal the comparison between the fast oscillations of the exact system, and the slow oscillations of the averaged system.

\begin{figure}[h!]\centering
	\includegraphics[width=\textwidth]{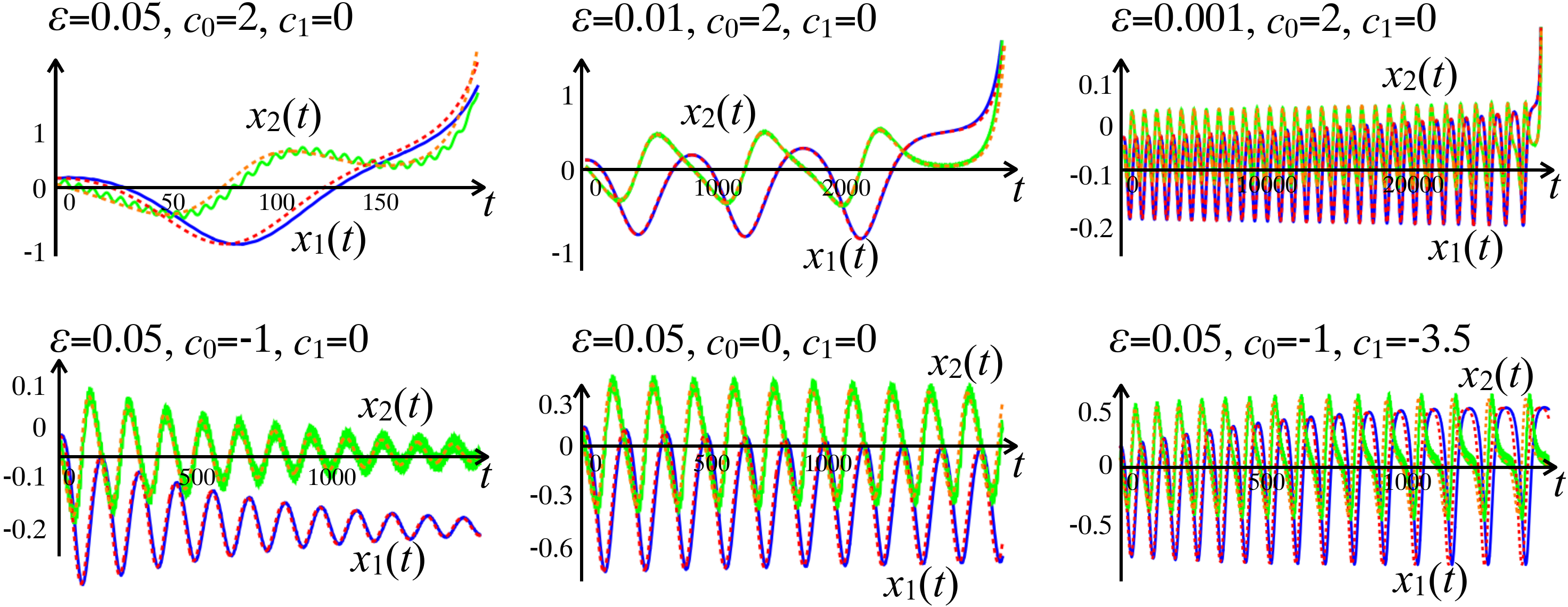}\vspace{-0.2cm}
	\caption{\sf Solutions of the system exhibiting a saddle-focus. Initial conditions $(x_1,x_2)=(0.1,0)$. In all cases we take $\mu=-0.2$ (as $\mu>0$ merely gives diverging solutions). Other parameters are given in the figure. Forward-time solutions shown only, showing the exact solutions for $x_1/x_2,$ components in blue/green, and averaged system $x_1/x_2,$ in dotted red/orange. The top row shows systems with unstable equilibria for $c_0=2$, $c_1=0$, and different $\e$ values. The bottom row shows system where the guiding system has for different $c_1$ and $c_0$ values (from left to right): a stable equilibrium, a centre, a stable limit cycle, respectively. 
	}\label{fig:sadfocus}
\end{figure}
\begin{figure}[h!]\centering
	\includegraphics[width=0.45\textwidth]{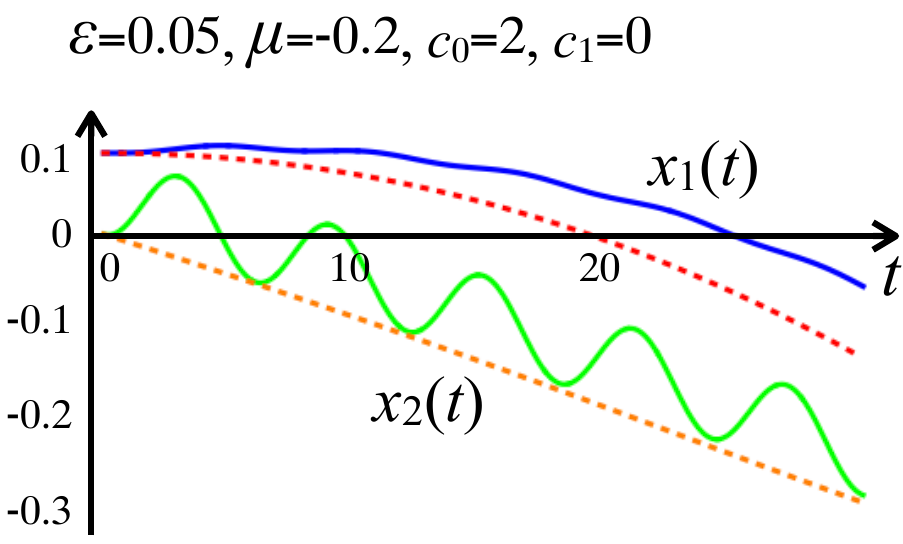}\vspace{-0.2cm}
	\caption{\sf Zoom in on a case of \cref{fig:sadfocus}, showing the difference between exact and averaged solution.
	}\label{fig:sadfocuszoom}
\end{figure}
%


\section{Preliminaries} \label{sec:prelim}

In this section, we introduce the basic framework necessary to prove in \Cref{sec:proofs} the theorems from \Cref{sec:results}. The reason for this introduction is twofold: it should help the reader's comprehension, avoiding long detours in multiple different references; and it should also establish notation and nomenclature. 

The material herein presented is already known, but we include proofs where the steps and notation are essential to the ensuing analysis.

\subsection{The Poincaré map and the displacement function of order $\ell$}

The use of displacement functions to study Poincaré maps has numerous examples in the literature (see, for instance, \cite{BUICA20047,CARMONA2023108501,CHICONE1991268,Candido_2017}). It reduces the problem of finding fixed points of a Poincaré map $\Pi$ to searching for zeroes of a displacement function $\Delta$. Recall that here our interest is in singular zeroes of the guiding system at the origin for $\mu=0$, that is, satisfying the conditions \cref{hypothesissingular1,hypothesissingular2} at the start of \Cref{sec:maintheorem}.

If $x(t,x_0,\mu,\e)$ denotes the solution of \cref{eq:completeaveragedsystemintro} satisfying $x(0,x_0,\mu,\e) = x_0$, then the family of stroboscopic Poincaré maps $\Pi$ is given by $\Pi(x_0,\mu,\e):=x(T,x_0,\mu,\e)$. The fact that this map is well-defined at least locally is guaranteed by the following result.
\begin{lemma} \label{lemma.timeTflow}
	Let $p \in D$ be given. There are a neighbourhood $U_p$ of $p$, a compact $K \subset \Sigma$ containing the origin, and $\e_M>0$ such that, for each $(x_0,\mu,\e) \in U_p \times K \times [-\e_M,\e_M]$, the solution $t\mapsto x(t,x_0,\mu,\e)$ of \cref{eq:completeaveragedsystemintro} is well-defined for $t \in [0,T]$. It is also smooth in $(x_0,\mu,\e)$ in that domain.
\end{lemma}
\begin{proof}
	Let $G(t,x,\mu,\e) = \sum_{i=1}^N \e^{i-1} g_i(x) + \e^{N} r_N (t,x,\mu,\e)$, so that \cref{eq:completeaveragedsystemintro} can be written $\dot x = \e G(t,x,\mu,\e)$. Choose $\delta>0$ such that the open ball $B_p(2\delta)$ centered at $p$ is contained in $D$, and let $K \subset \Sigma$ be a compact set containing the origin. Periodicity in time guarantees that $M:=\sup \{\|G(t,x,\mu,\e)\|: (t,x,\mu,\e) \in \R \times B_p(2\delta) \times K \times [-\frac{\e_0}{2},\frac{\e_0}{2}] \} $ is finite. Finally, choose $\e_M := \min \left\{ \frac{\e_0}{2}, \frac{\delta}{2 T M}\right\}$.
	
	Fix values of the parameters $\mu \in K$ and $\e \in [-\e_M,\e_M]$ and set $U_p:=B_p(\delta)$. By the Picard-Lindelöf Theorem, the solution $t \mapsto x(t,x_0,\mu,\e)$ of \cref{eq:completeaveragedsystemintro} with initial condition $x_0 \in B_p(\delta)$ is defined on a maximum interval of existence $I =(\omega^-(x_0), \omega^+(x_0)) \subset \R$, with $\omega^-(x_0) <0$ and $\omega^+(x_0)>0$. It is sufficient for us to prove that $\omega^+(x_0) > T$.
	
	By contradiction, assume that $\omega^+(x_0) \leq T$. This can only be if $x(t,x_0,\mu,\e)$ leaves $B_p(2\delta)$ at some $t_* \in (0,T]$, otherwise the solution would be well-defined by the Picard-Lindelöf Theorem. However, since $x(t,x_0,\mu,\e)$ is a solution of \cref{eq:completeaveragedsystemintro}, it follows that
	\begin{equation}
		\|x(t_*,x_0,\mu,\e) - x_0\| \leq \left\| \int_0^{t_*} \e G(x(s),x_0,\mu,\e) ds \right\| \leq \e_M T M \leq \frac{\delta}{2}. 
	\end{equation}
	Therefore, considering that $x_0 \in B_p(\delta)$, it follows at once that $x(t_*,x_0,\mu,\e) \in B_p(2\delta)$, which contradicts the definition of $t_*$. Smoothness follows from regularity of solutions of smooth differential equations with respect to initial conditions and parameters (see \cite[Chapter V]{Hartman1964OrdinaryDE}, for instance).
\end{proof}

We then introduce the displacement function $\Delta$, given by
\begin{equation}
\begin{aligned}
	\Delta(x_0,\mu,\e) &= \Pi(x_0,\mu,\e) - x_0 \nonumber\\
			&= \e^\ell \int_0^T g_\ell(x(\tau,x_0,\mu,\e),\mu) + \e R_\ell(\tau,x(\tau,x_0,\mu,\e),\mu,\e) d\tau.
\end{aligned}
\end{equation}
It is clear from \cref{lemma.timeTflow} that $\Delta$ is locally smooth near any point in $D$, and that zeroes of $\Delta$ correspond to fixed points of $\Pi$, which correspond to $T$-periodic solutions of \cref{eq:completeaveragedsystemintro} and, consequently, \cref{eq:standardsystemintro}. 
Since our interest lies essentially on the case $\e\neq0$, let us dispose of the $\e^\ell$ term in the displacement function and define $\Delta_\ell: \tilde{D} \times \Sigma \times (-\e_0,\e_0) \to \R^n$ by 
\begin{equation} \label{eq:displacementellintro}
	\Delta_\ell(x_0,\mu,\e) = \int_0^T g_\ell(x(\tau,x_0,\mu,\e),\mu) + \e R_\ell(\tau,x(\tau,x_0,\mu,\e),\mu,\e) d\tau.
\end{equation}

Throughout this paper, we will call $\Delta_\ell$ the \textit{displacement function of order $\ell$}. Once again, it follows from \cref{lemma.timeTflow} that $\Delta_\ell$ must be smooth near any chosen singular point in the domain for values of the parameters close to zero. The key step in our analysis is noticing that \cref{eq:displacementellintro} guarantees that $\Delta_\ell$ provides an unfolding of the germ of the map $x \mapsto T g_\ell(x,0)$. The notions of unfolding and germ will be summarized in \Cref{sec:germsandkequivalence}.

Observe that, by definition,
\begin{equation}\label{eq:poincareelldisplacementintro}
	\Pi(x_0,\mu,\e) = x_0 +   \e^\ell\Delta_\ell(x_0,\mu,\e).
\end{equation}
This identity is essential to our analysis as it relates the $\Pi$ Poincaré map and the displacement function $\Delta_\ell$ of order $\ell$. 
More precisely, it is easy to establish from \cref{eq:poincareelldisplacementintro} that $M_\Pi$ can be identified with the set $Z_{\Delta_\ell} \cup \{(x,\mu,0) : (x,\mu) \in D \times \Sigma \}$, where $Z_{\Delta_\ell}$ is the set of zeroes of $\Delta_\ell$. The study of $M_\Pi$ can then be performed by analysing the zeroes of the displacement function of order $\ell$.

\subsubsection{Fixed points and zeroes}

When referring to and classifying zeroes of families of functions, we use the following definition.  
\begin{definition}
	Let $U$ be an open subset of $\R^n$, $n \in \mathbb{N}$, and $V$ an open subset of $\R^k$, $k \in \mathbb{N}$, and $F:U \times V \to \R^m$ be any family of functions. We define
	\begin{enumerate}[label=(\alph*)]
		\item the zero set of $F$ by $Z_F = \{(x,\eta) \in U \times V: F(x,\eta) = 0\}$;
		\item the function $F_\eta: U \to \R^m$, where $\eta \in V$, by $F_\eta(x) = F(x,\eta)$;
		\item the zero set of $F_\eta$, where $\eta \in V$, by $Z_F(\eta)=\{x \in U: F(x,\eta)=0\}$.
	\end{enumerate}  
	Elements of $Z_F$ and $Z_F(\eta)$ are called, respectively, zeroes of $F$ and zeroes of $F_\eta$.
\end{definition}
We can then connect fixed points of $\Pi$ with zeroes of $\Delta_\ell$, as well as expressing $M_\Pi$ with the help of the set of zeroes of $\Delta_\ell$. The proofs of these follow directly from \cref{eq:poincareelldisplacementintro}.
\begin{proposition} \label{propositionfixedpointzero}
	Let $\mu \in \Sigma$ and $\e \in (-\e_0,\e_0) \setminus \{0\}$ be given. Then, $x \in D$ is a fixed point of $x \mapsto \Pi(x,\mu,\e)$ if, and only if, $(x,\mu,\e)$ is a zero of $\Delta_\ell$. 
\end{proposition}
\begin{corollary}\label{corollarycatastrophesurface}
	$M_\Pi=Z_{\Delta_\ell} \cup \{(x,\mu,0) : (x,\mu) \in D \times \Sigma \}.$
\end{corollary}

\bigskip\bigskip\bigskip
\subsection{Germs and $\mathcal{K}$-equivalence} \label{sec:germsandkequivalence}
\subsubsection{Germs}
An important concept in singularity theory is that of \textit{germs}. Essentially, a germ of an object captures only its only local properties. It is usually expressed as an equivalence class. We define that concept below, and set the notation we will be adopting throughout the remainder of the paper.

For convenience, we assume without loss of generality that the point near which our analysis is done will always be the origin, and we say that $U \in \mathcal{N}_0(S)$ if $U$ is an open neighbourhood of the origin contained in the set $S$. The first two definitions build the concept of germs of maps.
\begin{definition}
	Let $n,p \in \mathbb{N}$ and $U,U' \in \mathcal{N}_0(\R^n)$. Two maps $f:U \to \R^p$ and $g:U' \to \R^p$ are said to be germ-equivalent at the origin if there is $U'' \in \mathcal{N}_0(U \cap U')$ such that $f|_{U''} = g|_{U''}$.
\end{definition}
\begin{definition}
	Let $n,p \in \mathbb{N}$ and $U \in \mathcal{N}_0(\R^n)$. The germ of a map $f: U \to \R^p$ at the origin is the equivalence class $[f]$ of $f$ under germ-equivalence at the origin. The set of germs of functions from $\R^n$ to $\R^p$ at the origin is denoted by $\mathcal{E}_n^p$. If $p=1$, we usually simplify the notation to simply $\mathcal{E}_n$
\end{definition}
The set $\mathcal{E}_n^p$ is a vector space over $\R$ with the naturally induced operations of sum of functions and product of a function by a real number. $\mathcal{E}_n$ is itself a ring with the usual operations of sum and product of real-valued functions, and $\mathcal{E}_n^p$ can be also seen as the free module of rank $p$ over $\mathcal{E}_n$, i.e, any $[f] \in \mathcal{E}_n^p$ can be seen as a $p$-sized vector with entries in $\mathcal{E}_n$.

In singularity theory, it is usually useful to distinguish the special class of germs of diffeomorphisms as follows. 
\begin{definition} \label{def:germslocaldiff}
	Let $n \in \mathbb{N}$. $[\phi] \in \mathcal{E}_n^n$ is said to be the germ of a local diffeomorphism at the origin if there is one element $\phi$ in the class $[\phi]$ for which
	\begin{enumerate}
		\item $\phi(0) = 0$;
		\item $D\phi(0)$ is invertible.
	\end{enumerate} 
	The set of germs of local diffeomorphisms at the origin on $\R^n$ is denoted by $L_n$.
\end{definition}
Observe that $L_n$ is a group under the natural operation induced by composition. Moreover, the group $L_n$ acts on $\mathcal{E}_n^p$ on the right by the operation induced naturally by composition. 

From now on, we adopt the notation $[f]:(\R^n,0) \to (\R^p,0)$ to mean that $[f] \in \mathcal{E}_n^p$ and that $f(0)=0$. Equivalently, we may say that $[f] \in \mathcal{Z}_n^p$. We proceed now to the crucial concept of unfolding of a germ. 
\begin{definition}
	A $k$-parameter unfolding of a germ $[f]: (\R^n,0) \to (\R^p,0)$ is a germ $[\tilde{F}] : (\R^{n+k},0) \to (\R^{p+k},0)$ such that
	\begin{enumerate}
		\item a representative $\tilde{F}$ of $[\tilde{F}]$ is of the form $\tilde{F}(x,\eta) = (F(x,\eta),\eta)$;
		\item $F(x,0) = F_0(x) = f(x)$.
	\end{enumerate}
	The set of $k$-parameter unfoldings is denoted by $\mathcal{Z}_{n,k}^p$. More specifically, the set of $k$-parameter unfoldings of the identity in $\R^n$ is denoted by $L_{n,k}$.
\end{definition}

Finally, we present some algebraic definitions that will be useful when defining $\mathcal{K}$-equivalence.
\begin{definition}
	$\text{GL}_p(\mathcal{E}_n)$ is the set of $p \times p$ matrices $[M]$ with entries in $\mathcal{E}_n$ and for which $\det M(0) \neq0$.
\end{definition}
Any $[M] \in \text{GL}_p(\mathcal{E}_n)$ acts on $\mathcal{E}_n^p$ as matrix-vector multiplication with entries in $\mathcal{E}_n$.

\subsubsection{$\mathcal{K}$-equivalence}
The concept of $\mathcal{K}$-equivalence - also known as contact equivalence (see \cite{Golubitsky1985,Mather1968,Montaldi_2021}) or V-equivalence (see \cite{Martinet76}) - is part of the standard theory of singularities. We briefly introduce it here, confined to what is necessary to our discussion. The interested reader is referred to the more thorough presentation in \cite{Montaldi_2021}.

\begin{definition}
	Two germs $[f],[g] \in \mathcal{Z}_n^p$ are said to be $\bm{\mathcal{K}}$-\textbf{equivalent} if there are $[\phi] \in L_n$ and $[M] \in \text{GL}_p(\mathcal{E}_n)$ such that $[f] =[M] \cdot [g] \circ [\phi]$.
\end{definition}

The concept of $\mathcal{K}$-equivalence can also be used to study families of maps, through the ideas of $\mathcal{K}$-induction and $\mathcal{K}$-equivalent unfoldings, which are developed in the next definitions.
\begin{definition}
	Two unfoldings $[\tilde{F}],[\tilde{G}] \in \mathcal{Z}^p_{n,k}$ of the same germ $[f] \in \mathcal{Z}^p_n$ are said to be $\bm{\mathcal{K}}$\textbf{-isomorphic} if, for any representatives $\tilde{F}(x,\eta)=(F(x,\eta),\eta)$ and $\tilde{G}(x,\eta)=(G(x,\eta),\eta)$ of $[\tilde{F}]$ and $[\tilde{G}]$, respectively, there are a smooth function $\alpha$, defined on a neighbourhood $U \times V$ of the origin in $\R^n \times \R^k$ and satisfying $\alpha(x,0) = x$, and a smooth matrix function $Q$, also defined on $U \times V$ and satisfying $Q(x,0)=I_p$, such that the identity
	\begin{equation}
		F(x,\eta) = Q(x,\eta) \cdot G(\alpha(x,\eta),\eta)
	\end{equation}
	holds in $U \times V$.
\end{definition}

\begin{definition}
	Let $[\tilde{F}] \in \mathcal{Z}_{n,l}^p$ and $[h]:(\R^k,0) \to (\R^l,0)$. The pullback of $[\tilde{F}]$ by $[h]$, denoted by $[h]^*[\tilde{F}]$, is the unfolding $[\tilde{P}] \in \mathcal{Z}^p_{n,k}$ given by
	\begin{equation}
		\tilde{P}(x,\eta) = (F(x,h(\eta)),\eta).
	\end{equation}
\end{definition}

\begin{definition} \label{definitioncontactinduced}
	The unfolding $[\tilde{F}] \in \mathcal{Z}^p_{n,k}$ of the germ $[f] \in \mathcal{Z}_n^p$ is said to be $\bm{\mathcal{K}}$\textbf{-induced} by the unfolding $[\tilde{G}] \in\mathcal{Z}^p_{n,l}$ of the same germ via $[h]:(\R^k,0) \to (\R^l,0)$ if the unfoldings $[\tilde{F}]$ and $[h]^*[\tilde{G}]$ are $\mathcal{K}$-isomorphic. In other words, if $\tilde{F}(x,\eta) = (F(x,\eta),\eta)$, $\tilde{G}(x,\xi)=(G(x,\xi),\xi)$, and $h(\eta)$ are representatives of $[\tilde{F}]$, $[\tilde{G}]$, and $[h]$, respectively, there are neighbourhoods of the origin $U \subset \R^n$, $V \subset \R^k$, and smooth functions $\alpha:U \times V \to \R^n$ and $Q: U \times V \to \R^{p \times p}$ such that $h(0)=0$, $\alpha(x,0)=x$, $Q(x,0)=I_p$, and
	\begin{equation} \label{eq:definitionkinducedunfolding}
		F(x,\eta) = Q(x,\eta) \cdot G(\alpha(x,\eta),h(\eta))
	\end{equation}
	for $(x,\eta) \in U \times V$.
\end{definition}

\begin{definition} \label{definitioncontactequivalent}
	The unfolding $[\tilde{F}] \in \mathcal{Z}^p_{n,k}$ of the germ $[f] \in \mathcal{Z}_n^p$ is said to be $\bm{\mathcal{K}}$\textbf{-equivalent} to the unfolding $[\tilde{G}] \in\mathcal{Z}^p_{n,k}$ of the same germ if there is $[h] \in L_k$ such that the unfolding $[\tilde{F}]$ is $\mathcal{K}$-induced by $[\tilde{G}]$ via $[h]$.
\end{definition}

\subsubsection{Versality and codimension}
One of the central concepts of singularity theory is that of versality. In essence, an unfolding of a germ of a map is said to be versal if it induces all other possible unfoldings of the same germ. This means that all the possible unfoldings of that germ are codified in a versal unfolding, so that a versal unfolding carries all the information needed to unfold a germ. 

Naturally, different equivalence relations give rise to different notions of induction and equivalence, and thus also different notions of versality. In this paper, we are interested in $\mathcal{K}$-versality, i.e, versality with respect to $\mathcal{K}$-equivalence.

\begin{definition} \label{definitioncontactversal}
	The unfolding $[\tilde{F}] \in \mathcal{Z}^p_{n,k}$ of the germ $[f] \in \mathcal{Z}_n^p$ is said to be $\bm{\mathcal{K}}$-\textbf{versal}, if any other unfolding $[\tilde{G}] \in\mathcal{Z}^p_{n,k}$ of $[f]$ is $\mathcal{K}$-induced by $[\tilde{F}]$ via the germ of some mapping $[h]$.
\end{definition}

The theory of singularities provides a number of results aimed at verifying $\mathcal{K}$-versality, one of the its products being the important concept of codimension of a germ. Below, we present those results and the concepts involved. With the exception of \cref{propositionpushforwardinduction}, for which we provide a proof on account of its specificity, all other proofs are well-known and can be found, for instance, in \cite{Montaldi_2021}.

\begin{definition}
	Let $\bm{X}^0_n$ be the set of germs of $n$-dimensional vector fields at the origin having zero as equilibrium and $\bm{M}^0_p$ be the set of germs of matrix functions $\R^n \to \R^{p\times p}$ at the origin of $\R^n$. The extended $\mathcal{K}$-tangent space of a germ $[f] \in \mathcal{Z}_n^p$ is defined as the subspace of the vector field $\mathcal{E}_n^p$ given by
	\begin{equation}
		T_{\mathcal{\mathcal{K}},e} f := \{[Df] \cdot [X] + [M] \cdot [f]: [X] \in \bm{X}^0_n, [M] \in \bm{M}^0_p \}.
	\end{equation}
\end{definition}
\begin{definition}
	The $\mathcal{K}$-codimension of a germ $[f] \in \mathcal{Z}_{n}^p$, denoted by the symbol $\textnormal{codim}_\mathcal{K} ([f])$, is the codimension in $\mathcal{E}_n^p$ of the linear subspace $T_{\mathcal{\mathcal{K}},e} f$, or, which is the same, 
	\begin{equation}
		\textnormal{codim}_\mathcal{K} ([f]) = \dim\left(\mathcal{E}_n^p / T_{\mathcal{\mathcal{K}},e} f \right).
	\end{equation}
\end{definition}

Having defined codimension in the context of $\mathcal{K}$-equivalence, we proceed to stating two fundamental results relating versality and codimension.
\begin{proposition}
	Two $\mathcal{K}$-equivalent germs have the same $\mathcal{K}$-codimension. 
\end{proposition}

\begin{theorem} \label{theoremcontactversality}
	Let $[f] \in \mathcal{Z}_n^p$ be such that $\textnormal{codim}_\mathcal{K}([f])=d$. The following hold:
	\begin{enumerate}
		\item An unfolding $[\tilde{F}] \in \mathcal{Z}_{n,k}^p$ of $[f]$ with a representative of the form $$\tilde{F}(x,\eta_1,\ldots,\eta_k) = \left(F(x,\eta_1,\ldots,\eta_k),\eta_1,\ldots,\eta_k\right)$$  is $\mathcal{K}$-versal if, and only if,
		\begin{align*}
			T_{\mathcal{\mathcal{K}},e} f + \textnormal{span}_\mathbb{R} \left(\left[\frac{\partial F}{\partial \eta_1}\Big|_{\eta=0}\right],\ldots,\left[\frac{\partial F}{\partial \eta_k}\Big|_{\eta=0}\right]\right) = \mathcal{E}_n^p.
		\end{align*}
		\item There is $[\tilde{H}] \in \mathcal{Z}_{n,d}^p$ that is a $\mathcal{K}$-versal unfolding of $[f]$.
		\item If $[\tilde{F}],[\tilde{G}] \in \mathcal{Z}_{n,k}^p$ are $\mathcal{K}$-versal unfoldings of $[f]$, then they are $\mathcal{K}$-equivalent.
	\end{enumerate}
\end{theorem}

The concept of codimension is extremely important in the context of this paper because of its role defining the concept of $\mathcal{K}$-universality. 
\begin{definition}
	Let $[f] \in \mathcal{Z}_n^p$ be such that $\textnormal{codim}_\mathcal{K}([f])=d$. A $\mathcal{K}$-versal unfolding $[\tilde{F}] \in \mathcal{Z}_{n,k}^p$ of $[f]$ is said to be $\bm{\mathcal{K}}$-\textbf{universal} if $k=d$, that is, the number of parameters of the unfolding $[\tilde{F}]$ is equal to the codimension of $[f]$. 
\end{definition}

An important property of the codimension of a germ is established by \cref{propositioncodim}. \Cref{propositionuniversaliff}, on the other hand, characterizes $\mathcal{K}$-universal unfoldings.
\begin{proposition} \label{propositioncodim}
	Let $[f] \in \mathcal{Z}_n^p$ be such that $\textnormal{codim}_\mathcal{K}([f])=d$. Then, $d$ is the minimal number of parameters that an unfolding of $[f]$ must have to be $\mathcal{K}$-versal. 
\end{proposition}

\begin{proposition}\label{propositionuniversaliff}
	An unfolding $[\tilde{F}] \in \mathcal{Z}_{n+k}^p$ of $[f] \in \mathcal{Z}_{n^p}$ is $\mathcal{K}$-universal if, and only if,
	\begin{enumerate}
		\item $ \left\{\left[\frac{\partial F}{\partial \eta_1}\Big|_{\eta=0}\right],\ldots,\left[\frac{\partial F}{\partial \eta_k}\Big|_{\eta=0}\right]\right\} \subset \mathcal{E}_n^p $ is linearly independent;
		\item $	T_{\mathcal{\mathcal{K}},e} f \oplus \textnormal{span}_\mathbb{R} \left(\left[\frac{\partial F}{\partial \eta_1}\Big|_{\eta=0}\right],\ldots,\left[\frac{\partial F}{\partial \eta_k}\Big|_{\eta=0}\right]\right) = \mathcal{E}_n^p$.
	\end{enumerate}
\end{proposition}

We now define the pushforward of an unfolding, a concept that will be important for establishing the idea of equivalent families of two distinct but equivalent germs.
\begin{definition}
	Let $[\tilde{F}] \in \mathcal{Z}_{n,k}^p$ be an unfolding of the germ $[f] \in \mathcal{Z}_n^p$. For any given pair $([M],[\phi]) \in GL_p(\mathcal{E}_n)\times L_n$, the \textit{pushforward} of $[\tilde{F}]$ by $([M],[\phi])$, denoted by $([M],[\phi]) * [\tilde{F}] $ is defined as the unfolding of the germ $[f_{\text{push}}] = [M] \cdot [f] \circ [\phi]$ whose representative $\tilde{F}_{\text{push}}(x,\eta)=(F_{\text{push}}(x,\eta),\eta)$ satisfies
	\begin{equation}
		F_{\text{push}}(x,\eta) = M(x) \cdot F(\phi(x),\eta),
	\end{equation}
	in a neighbourhood of the origin.
\end{definition}

The pushforward has the important property of preserving induction, as proved in the next proposition.
\begin{proposition} \label{propositionpushforwardinduction}
	Let $[f] \in \mathcal{Z}_n^p$ and $([M],[\phi]) \in GL_p(\mathcal{E}_n)\times L_n$. The pushforward $([M],[\phi])*$ is a bijective map between the set of unfoldings of $[f]$ and unfoldings of $[f_{\text{push}}] = [M] \cdot [f] \circ [\phi]$ that preserves $\mathcal{K}$-induction.
\end{proposition}
\begin{proof}
	It is bijective because it has an inverse given by $([M^{-1}],[\phi^{-1}])$. 
	
	Suppose that $[\tilde{G}] \in \mathcal{E}_{n,l}^p$ is $\mathcal{K}$-induced by $[\tilde{F}]$ via $[h]$. Then, there are neighbourhoods of the origin $U \subset \R^n$, $V \subset \R^k$, and smooth functions $\alpha:U \times V \to \R^n$ and $Q: U \times V \to \R^{p \times p}$ such that $h(0)=0$, $\alpha(x,0)=x$, $Q(x,0)=I_p$, and
	\begin{equation} 
		G(x,\eta) = Q(x,\eta) \cdot F(\alpha(x,\eta),h(\eta))
	\end{equation}
	for $(x,\eta) \in U \times V$.	
	
	Define $[\tilde{F}_\text{push}] = ([M],[\phi]) * [\tilde{F}] $ and $[\tilde{G}_\text{push}] = ([M],[\phi]) * [\tilde{G}] $. By definition, if $\tilde{F}_{\text{push}}(x,\eta)=(F_{\text{push}}(x,\eta),\eta)$ and $\tilde{G}_{\text{push}}(x,\eta)=(G_{\text{push}}(x,\eta),\eta)$ are representatives, then 
	\begin{equation}
		G_{\text{push}} (x,\eta) = M(x)  \cdot Q(\phi(x),\eta) \cdot F(\alpha(\phi(x),\eta),h(\eta)).
	\end{equation}
	Setting $\beta(x,\eta) = \phi^{-1}(\alpha(\phi(x),\eta))$ and $S(x,\eta) = M(x) Q(\phi(x),\eta) M^{-1}(x)$, it follows that
	\begin{equation}
		G_{\text{push}} (x,\eta) =S(x,\eta) \cdot F_\text{push}(\beta(x,\eta),h(\eta)),
	\end{equation}
	which proves that $[\tilde{G}_\text{push}]$ is $\mathcal{K}$-induced by $[\tilde{F}_\text{push}]$ via $[h]$.
\end{proof}
\begin{corollary}
	The pushforward of $\mathcal{K}$-equivalent unfoldings are $\mathcal{K}$-equivalent. Also, the pushforward of a $\mathcal{K}$-versal unfolding is $\mathcal{K}$-versal. 
\end{corollary}

\subsection{The averaging method: a brief presentation}\label{sec:averagingpresentation}

A widely used technique for analysing non-linear oscillatory systems under small perturbations is the averaging method. This method has been rigorously formalized in a series of works starting with Fatou \cite{Fa} and including those by Krylov and Bogoliubov \cite{BK}, Bogoliubov \cite{Bo}, and later by Bogoliubov and Mitropolsky \cite{BM}. However, its origins can be traced back to the early perturbative methods applied in the study of solar system dynamics by Clairaut, Laplace, and Lagrange. For a concise historical overview of averaging theory, see \cite[Chapter 6]{MN} and \cite[Appendix A]{Sanders2007}.

In this section, we briefly introduce this method and provide explicit formulas for calculating the terms appearing in the definition of $\Delta_\ell$, the displacement function of order $\ell$.

\subsubsection{Main concepts}

Averaging theory is usually built around systems of the form \cref{eq:standardsystemintro}, which are said to be in the \textit{standard form}.  A key result of this theory concerns the possibility of moving time-dependency to higher orders of the parameter, and can be stated as follows (see \cite[Lemma 2.9.1]{Sanders2007} for a proof).

\begin{lemma} \label{lemmaaveragingtransform}
	There is a smooth near-identity map
	\begin{equation}\label{nit}
		X=U(t,z ,\mu,\e)= z+\sum_{i=1}^k \e^i\,  u_i(t,z,\mu ),
	\end{equation}
	$T$-periodic in $t$, that transforms differential system \eqref{eq:standardsystemintro} into
	\begin{equation}\label{fullaveq}
		z'=\sum_{i=1}^k\e^i g_i(z,\mu)+\e^{k+1} r_k(t,z,\mu,\e),
	\end{equation}
	with each of the functions on the right-hand side being smooth.
\end{lemma}

By imposing the existence of such a transformation and, then, solving homological equations, the functions $u_i$ and $g_i$, for $i\in\{1,2,\ldots,k\}$, can be recursively obtained. In general, 
\[
g_1(z,\mu)=\dfrac{1}{T}\int_0^T F_1(t,z,\mu)dt
\]
but, for $i\in\{2,\ldots,k\}$, such functions are not unique. Nevertheless, by imposing additionally the {\it stroboscopic condition} 
\[
U(z,0,\mu,\e)=z,
\]
each $g_i$ becomes uniquely determined, which is referred to as the {\it stroboscopic averaged function of order $i$} (or simply {\it $i$th-order averaged function})  of \eqref{eq:standardsystemintro}. If the original system is only of class $C^d$ for some $d \in \mathbb{N}^*$, there is generally loss of one derivative for each newly calculated order of the averaged functions.

The primary objective of the averaging method consists in estimating the solutions of the non-autonomous original differential equation \eqref{eq:standardsystemintro} by means of the following autonomous differential equation 
\begin{equation}\label{taq}
	z'=\sum_{i=1}^k\e^i  g_i(z,\mu),
\end{equation}
which corresponds to the truncation up to order $k$ in $\e$ of the differential equation \eqref{fullaveq}. Accordingly, for sufficiently small $|\e|\neq0$, the solutions of \eqref{eq:standardsystemintro} and  \eqref{taq}, with identical initial condition, remain $\e^k$-close over an interval of time of size $\mathcal{O}(1/\e)$ (see \cite[Theorem 2.9.2]{Sanders2007}).

The averaging method has shown to be highly effective in detecting the emergence of invariant structures of \eqref{eq:standardsystemintro} that originate from hyperbolic invariant structures of the autonomous differential equation
\begin{equation}\label{eq:gs}
	z'=g_{\ell}(z),
\end{equation}
where $g_{\ell}$ is the first averaged function that is not identically zero. As we have mentioned in \Cref{sec:results}, \eqref{eq:gs} is known as {\it guiding system}. A classical result of averaging theory in this context asserts the birth of an isolated periodic solution of \eqref{eq:standardsystemintro} provided that the guiding system \eqref{eq:gs} has a simple equilibrium (see, for instance, \cite{Hale,V}). This result has been extended to settings with less regularity \cite{BBLN19,BUICA20047,LMN15,LNR,LNT15,N22,NS21}. More recently, in \cite{NP24,PEREIRA20231}, this result has been generalized to detect higher-dimensional structures. Specifically, it has been proven that the differential equation \eqref{eq:standardsystemintro}  possesses an invariant torus, provided there is a hyperbolic limit cycle in the the guiding system \eqref{eq:gs} (see also \cite{CANDIDO20204555}).

\subsubsection{Calculation of the averaged functions}

As recently highlighted in  \cite{Novaes21b}, the averaging method is strictly related with the Melnikov method, which consists in expanding the solutions $X(t,X_0,\mu,\e)$ of \eqref{eq:standardsystemintro}, satisfying $X(0,X_0,\mu,\e)=X_0$, around $\e=0$ as (see \cite{LliNovTei2014,N17})
\begin{equation}\label{expansion}
	X(t,X_0,\mu,\e)=X_0+\sum_{i=1}^{k}\e^{i}\dfrac{y_i(t,X_0,\mu)}{i!}+\e^{k+1}r_{k}(t,X_0,\mu,\e), 
\end{equation}
where
\begin{equation}\label{yi}
	\begin{aligned}
		&y_1(t,X_0,\mu)= \int_0^tF_1(s,X_0,\mu)\,ds\\
		&y_i(t,X_0,\mu)= \int_0^t\bigg(i!F_i(s,X_0,\mu) + K_i (s, X_0, \mu) \bigg) ds, \,\, \text{ and }\vspace{0.3cm}\\
		&K_i(s,X_0,\mu) = \sum_{j=1}^{i-1}\sum_{m=1}^j\dfrac{i!}{j!}\p_x^m F_{i-j} (s,X_0,\mu)B_{j,m}\big(y_1,\ldots,y_{j-m+1}\big)(s,X_0,\mu),
	\end{aligned}
\end{equation}
for $i\in\{2,\ldots,k\}.$ Here, $B_{p,q}$ refers to the  {\it partial Bell polynomials} \cite{comtet}.  This formula can be easily implemented in algebraic softwares such as Mathematica and Maple.

From \eqref{expansion}, the stroboscopic Poincar\'{e} map becomes
\[
\Pi(X_0,\mu,\e)=X(T,X_0,\mu,\e)=X_0+\sum_{i=1}^k \e^i f_i(X_0,\mu)+\e^{k+1}R_k(X_0,\mu,\e)
\]
where $R_k(X_0,\mu,\e)=r_{k}(T,X_0,\mu,\e)$ and, for each $i$, 
\begin{equation}\label{avfunc}
	f_i(X_0,\mu)=\dfrac{y_i(T,X_0,\mu)}{i!}
\end{equation}
Notice that $ f_1=T g_1$. The  function $ f_i$ is referred to as the {\it Poincar\'{e}-Pontryagin-Melnikov function of order $i$}  or simply {\it $i$th-order Melnikov function} (and sometimes also {\it averaged functions}, see for instance \cite{LliNovTei2014}).

A formula connecting averaged and Melnikov functions was established in \cite[Theorem A]{Novaes21b}, given by
\begin{equation}\label{rec:gi}
	\begin{aligned}
		g_1(z,\mu)=&\dfrac{1}{T} f_1(z,\mu),\\
		g_i(z,\mu)=&\dfrac{1}{T}\left(  f_i(z,\mu)-\Theta(z,\mu)\right),
	\end{aligned}
\end{equation}
with $\tilde y_i(t,z,\mu),$ for $i\in\{1,\ldots,k\},$ being recursively defined by
\begin{equation}\label{tildeyi}
	\begin{aligned}
		\tilde y_1(t,z,\mu)=&t\,  g_1(z,\mu)\vspace{0.3cm}\\
		\tilde y_i(t,z,\mu)=& i!t\, g_i(z,\mu) +\Theta(z,\mu),
	\end{aligned}
\end{equation}
and with
\begin{equation}\label{gisum}
	\Theta(z,\mu)=\sum_{j=1}^{i-1}\sum_{m=1}^j\dfrac{i!}{j!}d^m  g_{i-j}(z,\mu) \int_0^t B_{j,m}\big(\tilde y_1,\ldots,\tilde y_{j-m+1}\big)(s,z,\mu)ds.
\end{equation}
This formula facilitates the calculation of the averaged functions without the need to handle the near-identity transformation \eqref{nit} and solving homological equations. For a practical implementation of this formula, we refer to \cite[Appendix A]{Novaes21b}, where a Mathematica algorithm is provided for computing the averaging functions.

As described in \Cref{sec:averagingintro}, the guiding system \eqref{eq:gs} plays a crucial role in averaging theory and is defined by the first averaged function that is not identically zero. In the following proposition, easily computable formulae are provided for this function, as well as for some of the subsequent averaged functions, using Melnikov functions.

\begin{proposition}\cite[Proposition 1]{Novaes21b-add}
	Let $\ell\in\{2,\ldots,k\}$. If either $ f_1=\cdots= f_{\ell-1}=0$ or $ g_1=\cdots= g_{\ell-1}=0,$ then
	\[
	g_i=\dfrac{1}{T}f_i,\quad \text{for}\quad i\in\{1,\ldots,2\ell-1\},
	\]
	and
	\[
	g_{2\ell}(z,\mu)=\dfrac{1}{T}\left(  f_{2\ell}(z,\mu)- \dfrac{1}{2} d  f_{\ell}(z,\mu)\cdot  f_{\ell}(z,\mu)\right).
	\]
\end{proposition}


\section{Proof of theorems} \label{sec:proofs}

This section contains the proofs for all the theorems stated in \Cref{sec:results}, as well as some novel auxiliary results that are used in those proofs.
\subsection{Auxiliary results: zero sets and $\mathcal{K}$-equivalence}
We first present and prove some useful (albeit slightly technical) results concerning the set of zeroes of different unfoldings of the same germ. The general idea of $\mathcal{K}$-equivalence preserving the zero sets of germs, and thus being useful in the study of bifurcations, is well-known (see \cite{Martinet76,Montaldi_2021}), we prove a precise formulation suited to our context here, to more easily consider strongly-fibred diffeomorphisms, maintaining the difference between the ``bifurcation" parameters $\mu$ and the perturbation parameter $\e$.
\begin{lemma}\label{theoremzeroesKequivalence}
	Let $[\tilde{F}], [\tilde{G}] \in\mathcal{Z}^p_{n,k}$ be unfoldings of, respectively, $[f], [g] \in \mathcal{Z}^p_n$. Assume that $[f]$ and $[g]$ are $\mathcal{K}$-equivalent and let $([M],[\phi]) \in GL_p(\mathcal{E}_n)\times L_n$ be such that $[g]=[M] \cdot [f] \circ[\phi]$. Also, let $\tilde{F}: \mathcal{D} \times \Sigma_k \to \R^p \times \R^k$, $\tilde{G}: \mathcal{D} \times \Sigma_k \to \R^p \times \R^k$ be representatives of $[\tilde{F}]$ and $[\tilde{G}]$ of the form $\tilde{F}(x,\eta)=(F(x,\eta),\eta)$ and $\tilde{G}(x,\eta)=(G(x,\eta),\eta)$. 
	
	If $[\tilde{G}]$ is $\mathcal{K}$-equivalent to $([M],[\phi])*[\tilde{F}]$ via $[h]$, there are $W \in  \mathcal{N}_0(\mathcal{D} \times \Sigma)$ and a diffeomorphism $\Phi:W \to E_\Phi$, satisfying $\Phi(x,\eta) =(\Phi_1(x,\eta),\Phi_2(\eta)) \in \R^n \times \R^k$, $\Phi(x,0)=(\phi(x),0)$, and
	\begin{equation} \label{eq:theoremzerosetequalityK}
		Z_F \cap E_\Phi= \Phi\left(Z_G \cap W\right).
	\end{equation}
	Additionally, if $F$ is independent of the last $k_F \in \{0,1,\ldots,k-1\}$ entries of $\eta = (\eta_1,\ldots,\eta_k)$ and $h=\left(h_1,\ldots,h_{k}\right)$ is such that
	\begin{equation} \label{eq:hypothesisdeterminantK}
		\det \left[ \frac{\partial \left(h_1,\ldots,h_{k-k_F}\right)}{\partial (\eta_1,\ldots,\eta_{k-k_F})}(0,0) \right]\neq0;
	\end{equation}
	then $\Phi_2$ can be chosen as $\Phi_2(\eta) =(h_1(\eta),\ldots,h_{k-k_F}(\eta),\eta_{k-k_F+1},\ldots,\eta_k)$. In particular, $\Phi_2=h$ can be chosen regardless of $F$.
\end{lemma}
\begin{proof}
	Since $[\tilde{G}]$ is $\mathcal{K}$-equivalent to $([M],[\phi])*[\tilde{F}]$ via the local diffeomorphism germ $[h]$, then a representative $h:\Sigma'_k \to \Sigma_k'$ is such that $h(0)=0$ and $Dh(0)$ is invertible. We can assume that $\Sigma_k' \in \mathcal{N}_0(\Sigma_k)$ is sufficiently small to ensure that $Dh(\eta)$ is invertible on $\Sigma_k'$. Moreover, there are $U_0 \in \mathcal{N}_0(\mathcal{D})$, $V_0 \in \mathcal{N}_0(\Sigma'_k)$, and smooth functions $Q(x,\eta)$ and $\alpha(x,\eta)$ such that $Q(x,0)=I_p$, $\alpha(x,0)=x$, and
	\begin{equation}\label{eq:equivalentKunfoldingmain}
		G(x,\eta) = Q(x,\eta) M(\alpha(x,\eta))F(\phi(\alpha(x,\eta)),h(\eta))
	\end{equation}
	for any $(x,\eta) \in U_0 \times V_0$. Without loss of generality, we assume that $\overline{U}_0 \subset \mathcal{D}$
	
	Since $\alpha$ and $Q$ are smooth, $\alpha(x,0)=x$, and $Q(x,0)=I_p$, we can find $U_1 \in \mathcal{N}_0(U_0)$ and $V_1 \in \mathcal{N}_0(V_0)$ sufficiently small as to guarantee that $D\alpha_{\eta}(x)$ and $Q(x,\eta)$ are invertible for $(x,\eta) \in U_1 \times V_1$ and that $\alpha(U_1 \times V_1)$ is in contained in a set where $M$ and $\phi$ are invertible. 
	
	Define $W = U_1 \times V_1$ and $\Phi(x,\eta) = (\phi(\alpha(x,\eta)),h(\eta))$, which is clearly of the desired form. Since
	\begin{equation}
		\det D\Phi(x,\eta) = \det D\phi(\alpha(x,\eta)) \cdot \det D\alpha_\eta (x) \cdot \det Dh(\eta),
	\end{equation}
	it follows that $D\Phi(x,\eta)$ is invertible for $(x,\eta) \in W$. Hence, $\Phi$ is a diffeomorphism on $W$, and it is easy to see that $\Phi(x,0)=(\phi(x),0)$. Let $E_\Phi$ be the image $\Phi(W)$.
	
	For the relationship between the $Z_F$ and $Z_G$, observe that, on the one hand, if $(x,\eta) \in Z_G \cap W$, then, by \cref{eq:equivalentKunfoldingmain},
	\begin{equation}
		F(\Phi(x,\eta)) = F(\phi(\alpha(x,\eta)),h(\eta)) = \left(M(\alpha(x,\eta))\right)^{-1} \,\left(Q(x,\eta)\right)^{-1} \, G(x,\eta) = 0,
	\end{equation}
	so that $\Phi(x,\eta) \in Z_G \cap E_\Phi$. On the other hand, if $(y,\xi) \in Z_F \cap E_\Phi$, \cref{eq:equivalentKunfoldingmain} ensures that $(x,\eta) := \Phi^{-1} (y,\xi)$ satisfies \begin{equation}
		G(x,\eta) = Q(x,\eta) M(\alpha(x,\eta)) F(\Phi(x,\eta)) =Q(x,\eta) G(y,\xi) = 0,
	\end{equation}
	so that $\Phi^{-1}(y,\xi) \in Z_G \cap W$. This proves \cref{eq:theoremzerosetequalityK}.
	
	Suppose now that the additional hypotheses of \cref{theoremzeroesKequivalence} hold, that is, $F$ is independent of the last $k_F<k$ entries of $\eta$ and \cref{eq:hypothesisdeterminantK} is valid. Define 
	\begin{equation}
		h_{\text{Fib}}(\eta)=\left(h_1(\eta),\ldots,h_{k-k_F}(\eta),\eta_{k-k_F+1},\ldots,\eta_k\right),
	\end{equation}
	
	It is easy to see that \cref{eq:hypothesisdeterminantK} ensures $h_{\text{Fib}}$ is a local diffeomorphism near the origin. Moreover, the independence of $F$ with respect to its last $k_F$ entries guarantees that  \cref{eq:equivalentKunfoldingmain} still holds after replacing $h$ with $h_{\text{Fib}}$.
	
	By retracing the steps of the proof with this new $h_{\text{Fib}}$, we obtain the analogous of \cref{eq:theoremzerosetequalityK} with $\Phi$ replaced by $\Phi_{\text{Fib}}(x,\eta) = (\phi(\alpha(x,\eta)),h_{\text{Fib}}(\eta))$, which is clearly of the desired form.
\end{proof}
\begin{remark}
	Two unfoldings $[\tilde{F}], [\tilde{G}] \in\mathcal{Z}^p_{n,k}$ of $\mathcal{K}$-equivalent germs satisfying the hypotheses of \cref{theoremzeroesKequivalence} are said to be $\mathcal{K}$-equivalent as families, even though it should be kept in mind that they are not unfoldings of the same germ, and thus cannot be considered equivalent unfoldings.
\end{remark}

The hypothesis of independence of $F$ with respect to the last entries of $\eta$ may seem quite arbitrary at a first glance, and thus merits an explanation. It is motivated by our implicit technical assumption in stating the lemma that $F$ and $G$ unfold equivalent germs \textit{using the same number $k$ of parameters.} Frequently that is not true, in which case we equivalently say that the unfolding with less parameters is independent of a number of its entries.

In fact, in order to prove \cref{maintheorem}, we will make use of the unfolding $\Delta_\ell(x,\mu,\e)$ of the singular germ $x\mapsto g_\ell(x,0)$ with one extra parameter ($\e$) than strictly required by its codimension. An application of \cref{theoremzeroesKequivalence} with the special form of $\Phi_2$ that explicitly maintains the separation of $\e$ from the rest of the parameters - i.e., choosing $\Phi$ as a strongly fibred diffeomorphism - is what then allows us to compare the catastrophe surface $M_\Pi$ with the suspension of the catastrophe surface of $Z_{g_\ell}$, as will be explained in \Cref{sec:proofmain}.

The next result is used to connect our hypothesis of $\mathcal{K}$-universality to the hypotheses of \cref{theoremzeroesKequivalence}. It is an important technical step in proving the main theorem of this paper, \cref{maintheorem}.
\begin{lemma} \label{propositioninduceduniversal}
	Let $[f] \in \mathcal{Z}_{n}^p$ a germ of $\mathcal{K}$-codimension $d$, and $[\tilde{H}] \in \mathcal{Z}_{n,d}^p$ be a $\mathcal{K}$-universal unfolding of $[f]$. Also, let $k\geq 0$ and $[\tilde{F}] \in \mathcal{Z}_{n,\,d+k}^p$ be an unfolding of $[f]$. Take $\tilde{H}:\mathcal{D} \times \Sigma_d \to \R^p \times \R^d$ and $\tilde{F}:\mathcal{D} \times \Sigma_{d+k} \to \R^p \times \R^{d+k}$ to be representatives of the form $\tilde{H}(x,\eta) = (H(x,\eta),\eta)$, $\tilde{F}(x,\eta,\xi)=(F(x,\eta,\mu),\eta,\xi)$, and assume that $F(x,\eta,0)=H(x,\eta)$. 
	
	Suppose that $[\tilde{F}]$ is $\mathcal{K}$-induced by $[\tilde{H}]$ via $[h]:(\R^{d+k},0) \to (\R^d,0)$. Then, 
	\begin{equation}
		\det \left(\frac{\partial h}{\partial \eta}(0,0)\right) \neq 0.
	\end{equation}
\end{lemma}
\begin{proof}
	Since $[\tilde{F}]$ is $\mathcal{K}$-induced by $[\tilde{H}]$ via $[h]$, it follows that there are $U_0 \in \mathcal{N}_0(\mathcal{D})$, $V_0 \in \mathcal{N}_0(\Sigma_{d+k})$, and smooth functions $Q(x,\eta,\xi)$ and $\alpha(x,\eta,\xi)$ such that $Q(x,0,0)=I_p$, $\alpha(x,0,0)=x$, and
	\begin{equation} \label{eq:propostionuniversal1}
		F(x,\eta,\xi) = Q(x,\eta,\xi) \cdot H(\alpha(x,\eta,\xi),h(\eta,\xi)).
	\end{equation}
	
	By hypothesis, $F(x,\eta,0)=H(x,\eta)$, so that we obtain from \cref{eq:propostionuniversal1} that
	\begin{equation}
		H(x,\eta) = Q(x,\eta,0) \cdot H(\alpha(x,\eta,0),h(\eta,0)).
	\end{equation}
	For each $x$, we have an identity of smooth functions of $\eta$, which can thus be differentiated at $\eta=0$ in the direction of $w \in \R^d$, yielding
	\begin{equation}
		\begin{aligned} \label{eq:propositionuniversal2}
			\frac{\partial H}{\partial \eta}(x,0) \cdot w =& \left(\frac{\partial Q}{\partial \eta} (x,0) \cdot w \right)\cdot H(x,0) \\ &+ \left(\frac{\partial H}{\partial x}(x,0) \cdot \frac{\partial \alpha}{\partial \eta}(x,0,0) + \frac{\partial H}{\partial \eta}(x,0) \cdot \frac{\partial h}{\partial \eta}(0,0)\right) \cdot w .
		\end{aligned}
	\end{equation}
	
	Let $w \in \R^d$ be given such that	
	\begin{equation}
		\frac{\partial h}{\partial \eta}(0,0) \cdot w = 0.
	\end{equation}
	We will show that $w=0$, so that the derivative of $h$ with respect to $\eta$ at $(0,0)$ must be invertible. From \cref{eq:propositionuniversal2}, it follows that
	\begin{equation} \label{eq:propositionuniversal3}
		\frac{\partial H}{\partial \eta}(x,0) \cdot w = \left(\frac{\partial Q}{\partial \eta} (x,0) \cdot w \right)\cdot H(x,0) + \frac{\partial H}{\partial x}(x,0) \cdot \frac{\partial \alpha}{\partial \eta}(x,0,0) \cdot w
	\end{equation}
	Observe that, since $H(x,0) = f(x)$, the right-hand side of \cref{eq:propositionuniversal3} is an element of the extended $\mathcal{K}$-tangent space $T_{\mathcal{K},e} f$. Moreover, it is clear that the left-hand side of the same identity belongs to the subspace
	\begin{equation}
		\textnormal{span}_\mathbb{R} \left(\left[\frac{\partial H}{\partial \eta_1}\Big|_{\eta=0}\right],\ldots,\left[\frac{\partial H}{\partial \eta_d}\Big|_{\eta=0}\right]\right).
	\end{equation}
	
	Considering that, by hypothesis, $[\tilde{H}]$ is a $\mathcal{K}$-universal unfolding of $f$, we know by \cref{propositionuniversaliff} that
	\begin{equation}
		T_{\mathcal{K},\e}f \cap \textnormal{span}_\mathbb{R} \left(\left[\frac{\partial H}{\partial \eta_1}\Big|_{\eta=0}\right],\ldots,\left[\frac{\partial H}{\partial \eta_d}\Big|_{\eta=0}\right]\right) = \{0\}.
	\end{equation}
	Thus, it follows at once that
	\begin{equation}
		\frac{\partial H}{\partial \eta}(x,0) \cdot w = 0.
	\end{equation}
	If $w \neq0$, then there would be a non-trivial linear combination of elements of 
	\begin{equation}
		\left\{\left[\frac{\partial H}{\partial \eta_1}\Big|_{\eta=0}\right],\ldots,\left[\frac{\partial H}{\partial \eta_d}\Big|_{\eta=0}\right]\right\}
	\end{equation}
	that vanishes, contradicting the linear independence of this family established in \cref{propositionuniversaliff}. 
	
	Therefore, it follows that $w=0$, concluding the proof.
\end{proof}

\subsection{Persistence of catastrophes: Proof of \cref{maintheorem}} \label{sec:proofmain}
Having proved the auxiliary results above, the proof of our main result, \cref{maintheorem}, is as follows.

By hypothesis, $[\tilde{H}] \in \mathcal{Z}_{n,k}^n$, defined by $\tilde{H}(x,\mu)=(H(x,\mu),\mu)=(g_\ell(x,\mu),\mu)$ is a $\mathcal{K}$-universal unfolding of the the germ $[s] \in \mathcal{Z}_n^n$ given by $s(x)=g_\ell(x,0)$. In particular, the unfolding $[\tilde{F}] \in \mathcal{Z}_{n, \, k+1}^n$ defined by $\tilde{F}(x,\mu,\e) = (F(x,\mu,\e),\mu,\e) = (\Delta_\ell(x,\mu,\e),\mu,\e)$ is $\mathcal{K}$-induced by $[\tilde{H}]$. Hence, let $Q(x,\mu,\e)$, $\alpha(x,\mu,\e)$ and $h(\mu,\e)$ be such that 
\begin{equation} \label{eq:maintheorem1}
	F(x,\mu,\e) = Q(x,\mu,\e) \cdot H(\alpha(x,\mu,\e),h(\mu,\e))
\end{equation}

It is easy to see that, since $\Delta_\ell(x,\mu,0)=g_\ell(x,\mu)$, it follows that $F(x,\mu,0) = H(x,\mu)$. Thus, all the hypotheses of \cref{propositioninduceduniversal} are valid, ensuring that
\begin{equation}
	\det \left(\frac{\partial h}{\partial \mu}(0,0)\right) \neq 0.
\end{equation}

Define $[\tilde{G}] \in \mathcal{Z}_{n,\,k+1}^n$ by $\tilde{G}(x,\mu,\e) = (G(x,\mu,\e),\mu,\e)=(g_\ell(x,\mu),\mu,\e)$. In particular, we have that $G(x,\mu,\e) = H(x,\mu)$. Hence, \cref{eq:maintheorem1} ensures that
\begin{equation} \label{eq:maintheorem2}
	F(x,\mu,\e) = Q(x,\mu,\e) \cdot G(\alpha(x,\mu,\e),h_{\textnormal{ex}}(\mu,\e)),
\end{equation}
where $h_{\textnormal{ex}}(\mu,\e) = (h(\mu,\e),\e)$, which is clearly a local diffeomorphism near the origin of $\R^{k+1}$. Therefore, $[\tilde{F}]$ is $\mathcal{K}$-equivalent to $[\tilde{G}]$ via $[h_{\text{ex}}]$.

Finally, an application of \cref{theoremzeroesKequivalence} guarantees the existence of a diffeomorphism $\Phi:U \to V$, satisfying $\Phi(x,\mu,\e) =(\Phi_1(x,\mu,\e),\Phi_2(\mu,\e),\e) \in \R^n \times \R^k \times \R$, $\Phi(x,0,0)=(x,0,0)$, and
\begin{equation}  \label{eq:maintheorem3}
	Z_F \cap V= \Phi\left(Z_G \cap U\right).
\end{equation}

By definition of $G$, it is clear that $Z_G \cap U = (Z_{g_\ell} \times \R) \cap U$. Similarly, $Z_F \cap V = Z_{\Delta_\ell} \cap V$. Thus, considering \cref{corollarycatastrophesurface}, it follows that $M_{\Pi} \cap V = \Phi \left((Z_{g_\ell} \times \R) \cap U \right) \cup V_{\e=0}$. The fact that $Z_{g_\ell} \times \{0\}$ is invariant under $\Phi$, follows from intersecting both sides of \cref{eq:maintheorem3} with the set $\{(x,\mu,0) \in \R^3 \}$, because the last coordinate function of $\Phi$ is $\e$ identically. In fact, by doing so, we obtain 
\begin{equation}
	\left(Z_{g_\ell} \times \{0\}\right) \cap V = \Phi\left(\left(Z_{g_\ell} \times \{0\}\right) \cap U\right),
\end{equation} 
proving the invariance.

\subsection{Persistence of bifurcation diagrams: proof of \cref{theorembifdiagram}}\label{sec:proofbif}
In this section, we make use of \cref{maintheorem} to prove the \cref{theorembifdiagram}, concerning the persistence of bifurcation diagrams of equilibria.

Observe that $\mathcal{D}_{\ell,0}$ is defined by $\Delta_\ell(x,\mu,0)=0$ and $\mathcal{D}_\e$ by $\Delta_\ell(x,\mu,\e)=0$. The fact that $g_\ell(x,\mu)=\Delta_\ell(x,\mu,0)$ is $\mathcal{K}$-universal ensures that it is a submersion near $(0,0)$ (for a proof of this fact, see \cite[Proposition 14.3]{Montaldi_2021}). Thus, by smoothness with respect to $\e$, it follows that, for small fixed $\e\neq0$, $(x,\mu) \mapsto \Delta_\ell(x,\mu,\e)$ is also a submersion near the origin. Hence, $\mathcal{D}_{\ell,0}$ and $\mathcal{D}_\e$ are smooth manifolds of codimension $k$ by the Regular Value Theorem.

The fact that $\mathcal{D}_\e$ is $\mathcal{O}(\e)$-close to $\mathcal{D}_{\ell,0}$ follows from \cref{maintheorem}. In fact, since $\mathcal{D}_\e$ can be obtained, for $\e \neq 0$, by intersecting $M_\Pi$ with the hyperplane attained by fixing $\e$, it follows that that $\mathcal{D}_\e$ is given by the image under $\Phi$ of $Z_{g_\ell} \times \{\Phi_3^{-1}(\e)\}$. Thus, if $\e' := \Phi_3^{-1}(\e)$,
\begin{equation}
	\mathcal{D}_\e = \{(\Phi_1(x,\mu,\e'),\Phi_2(\mu,\e')) : (x,\mu,\e') \in \left(Z_{g_\ell} \times \R \cap U\right) \}
\end{equation}
Considering that, by definition, $\e' = \mathcal{O}(\e)$ and that $\Phi$ is smooth, if follows that $\mathcal{D}_\e$ is $\mathcal{O}(\e)$-close to 
\begin{equation}
	\{(\Phi_1(x,\mu,0),\Phi_2(\mu,0)) : (x,\mu) \in Z_{g_\ell} \},
\end{equation}
which coincides with $Z_{g_\ell} = \mathcal{D}_{\ell,0}$ by the invariance statement of \cref{maintheorem}. This concludes the proof.


\subsection{Proof of stabilisation of non-stable families: the transcritical case}\label{sec:nonstabtran}

For a 1-dimensional vector field, the transcritical bifurcation is generally described as occurring in a 1-parameter family, as two equilibria collide and pass through each other, exchanging their stability properties. A normal form for the transcritical bifurcation is $\dot x = \mu x + x^2$.

Families displaying such behaviour are not stable, in that a small perturbation generally changes the phase portraits and breaks the bifurcation. However, they are still studied because they appear typically in 1-parameter families displaying a fairly common property: existence of an equilibrium for every value of the parameter (see \cite[Section 3.4]{Guckenheimer1983}).  

Let us begin with a definition of the transcritical bifurcation based on the concept of $\mathcal{K}$-equivalence.

\begin{definition} \label{def:transcriticaln1}
	A 1-parameter family of $1$-dimensional vector fields $F(x,\mu)$ is said to undergo a transcritical bifurcation at the origin for $\mu=0$ if
	\begin{enumerate}
		\item The germ of $f: x \mapsto F(x,0)$ at the origin is $\mathcal{K}$-equivalent to the germ of $s_{1,0}(x)=x^2$.
		\item  Let $([M],[\phi]) \in GL_n(\mathcal{E}_n)\times L_n$ be such that $[f]=[M] \cdot [s_{1,0}] \circ [\phi]$. The pushforward $([M],[\phi]) *[\tilde{\mathcal{U}}]$ of the unfolding $[\mathcal{U}] \in \mathcal{Z}_{1,1}^1$, given by $\tilde{\mathcal{U}}(x,\mu)=(\mathcal{U}(x,\mu),\mu)$ and $\mathcal{U}(x,\mu)=\mu x + x^2$, is $\mathcal{K}$-equivalent to $ [\tilde{F}]$ via the identity, where  $\tilde{F}(x,\mu)=(F(x,\mu),\mu)$.
	\end{enumerate}
\end{definition}
The definition essentially states that a transcritical family is characterized by a singularity whose unfolding is, up to $\mathcal{K}$-equivalence, given by the normal form $x \mapsto \mu x - x^2$. We now consider what happens when a transcritical bifurcation occurs in a guiding system. 

The important observation is that the normal form $\mu x +x^2$ of the transcritical can be `embedded' into the versal family $\lambda + z^2$ of the fold, by taking $(z(x,\mu),\lambda(\mu)) = (x+\mu/2,-\mu^2/4)$. \Cref{theorembifdiagram} can then be applied to the versal family $\lambda+y^2$, so that the bifurcation diagram of periodic orbits must be given by zeros of an $\mathcal{O}(\e)$-perturbation of it. In essence the possible bifurcation diagrams for a fixed $\e \neq0$ are given by $\eta(\e)+\lambda+y^2=0$, which, returning to the original coordinates, is $\eta(\e) +\mu x+ x^2 $. One can check that two different diagrams emerge depending on the sign of $\eta(\e)$. Namely, two nearby folds if $\eta>0$ and two approaching zeros that suffer no bifurcation if $\eta<0$.
\subsubsection{The canonical form of the displacement function}

\begin{proposition} \label{propositiontransn1}
	Let $n=k=1$ and suppose that the guiding system $\dot x = g_\ell (x,\mu)$ undergoes a transcritical bifurcation at the origin for $\mu=0$.
	Then, there are $\e_1 \in (0,\e_0)$, an open interval $I$ containing $0 \in \R$, an open neighbourhood $U_\Sigma \subset \Sigma$ of $0$, and smooth functions $\zeta,Q: I\times U_\Sigma \times (-\e_1,\e_1) \to \R$, $a,S: U_\Sigma \times (-\e_1,\e_1) \to \R$, and $b:(-\e_1,\e_1) \to \R$ such that
	\begin{enumerate}[label=(T.\Roman*)]
		\item \label{propertytransmain} If $\Delta_\ell$ is the displacement function of order $\ell$ of \cref{eq:standardsystemintro}, then 
		\begin{align*}
			\Delta_\ell(x,\mu,\e) = Q(x,\mu,\e)\left( \zeta^2(x,\mu,\e) + S(\mu,\e) a^2(\mu,\e)  + S(\mu,\e) b(\e) \right)
		\end{align*}
		for $(x,\mu,\e) \in I \times U_\Sigma \times (-\e_1,\e_1)$.
		\item \label{propertytransphi} For each $(\mu,\e) \in U_\Sigma \times (-\e_1,\e_1)$, the map $\zeta_{(\mu,\e)}:x \mapsto \zeta(x,\mu,\e)$ is a diffeomorphism on the interval $I$.
		\item \label{propertytransa} For each $\e \in(-\e_1,\e_1)$, $a_\e: \mu \mapsto a(\mu,\e)$ is a diffeomorphism on $U_\Sigma$.
		\item \label{propertytranseq}$b(0)=0$, $a(0,0) =0$, $\zeta(0,0,0)=0$, and $\sgn \left(Q(0,0,0)\right) = \sgn \left(\frac{\partial^2 g_\ell}{\partial x^2}(0,0)\right)$.
		\item \label{propertytransS} $S(\mu,\e)<0$ for any $(\mu,\e) \in U_\Sigma \times (-\e_1,\e_1)$.
	\end{enumerate}  
\end{proposition}
\begin{proof}
	We begin by observing that $\Delta_\ell(x,\mu,0) = T g_\ell(x,\mu)$, by definition of the displacement function of order $\ell$. Let $[s_{1,0}]$ be as in \cref{def:transcriticaln1} and $[\tilde{F}]$ be the 2-parameter unfolding of $f: x \mapsto Tg_\ell(x,0)$ given by $\tilde{F}(x,\mu,\e) = (\Delta_\ell(x,\mu,\e),\mu,\e)$. By hypothesis, there are $P(x,\mu) \in \R$ and $\psi(x,\mu) \in \R$ such that
	\begin{equation}\label{eq:transcriticalcanonical1}
		\Delta_\ell(x,\mu,0) = P(x,\mu) \left(\mu \psi(x,\mu) + \psi^2(x,\mu)\right)
	\end{equation}
	and, defining $M(x):=P(x,0)$, and $\phi(x):=\psi(x,0)$, it holds that $[f] = [M] \cdot [s_{1,0}] \circ [\phi]$. 
	
	Let $\tilde{H}(x,\eta)=(y^2+\eta,\eta) \in \R \times \R$. Since the 1-parameter unfolding $[\tilde{H}]$ of $[s_{1,0}]$ is $\mathcal{K}$-versal, it follows that $[\tilde{F}]$ must be $\mathcal{K}$-induced by $([M],[\phi]) * [\tilde{H}]$. Hence, there is a neighbourhood $\tilde{V}_1 : = I \times (-\tilde{\mu}_1,\tilde{\mu}_1) \times (-\tilde{\e}_1,\tilde{\e}_1)$ of the origin in $\R^{1+1+1}$ and smooth functions $h(\mu,\e) \in \R$, $Q(x,\mu,\e) \in \R$, and $\zeta(x,\mu,\e) \in \R $ such that $h(0,0) = 0$, $Q(x,0,0) = M(x) \neq 0$, $\zeta(x,0,0) = \phi(x)$, and
	\begin{equation}\label{eq:transcriticalcanonical2}
		\Delta_\ell(x,\mu,\e) = Q(x,\mu,\e) \cdot \left(\zeta^2(x,\mu,\e)+h(\mu,\e)\right).
	\end{equation}
	Because $\zeta(x,0,0)=\phi(x)$ is a local diffeomorphism, assuming that $\tilde{\mu}_1$ and $\tilde{\e}_1$ are sufficiently small, we can ensure that $\zeta_{(\mu,\e)}:x \mapsto \zeta(x,\mu,\e)$ is a diffeomorphism on $I$ for any $(\mu,\e) \in (-\tilde{\mu}_1,\tilde{\mu}_1) \times (-\tilde{\e}_1,\tilde{\e}_1)$.
	
	Since $[f] = [M]\cdot [s_{1,0}] \circ [\phi]$, it follows by twice differentiating at the origin that 
	\begin{equation}
		T \frac{\partial^2 g_\ell}{\partial x^2}(0,0) = 2 M(0) \left(\phi'(0)\right)^2.
	\end{equation} 
	Considering that $M(x)=Q(x,0,0)$, we obtain 
	\begin{equation}
		\sgn \left(\frac{\partial^2 g_\ell}{\partial x^2}(0,0)\right) = \sgn (M(0)) = \sgn \left(Q(0,0,0)\right) \neq 0.
	\end{equation}
	
	A combination of \cref{eq:transcriticalcanonical1,eq:transcriticalcanonical2} yields
	\begin{equation} \label{eq:transcriticalcanonicalmain}
		P(x,\mu) \left(\mu \psi(x,\mu) + \psi^2(x,\mu)\right) = Q(x,\mu,0) \cdot \left(\zeta^2(x,\mu,0)+h(\mu,0)\right)
	\end{equation}
	
	Differentiating both sides of \cref{eq:transcriticalcanonicalmain} with respect to $\mu$ at the origin and considering that $\psi(0,0) = \zeta(0,0,0)= \phi(0) = 0$, it follows that
	\begin{equation}\label{eq:transcriticalcanonicaldhdmu}
		\frac{\partial h}{\partial \mu}(0,0) = 0.
	\end{equation}
	
	Now, differentiating both sides of \cref{eq:transcriticalcanonicalmain} twice with respect to $x$ at the origin and considering that $M(0)$ is invertible, we obtain
	\begin{equation}\label{eq:transcriticalcanonicaldxx}
		\left(\frac{\partial \psi}{\partial x}(0,0)\right)^2 =  \left(\frac{\partial \zeta}{\partial x}(0,0,0)\right)^2.
	\end{equation}
	Partial differentiation with respect to $x$ and $\mu$ yields
	\begin{equation}\label{eq:transcriticalcanonicaldmux}
		\frac{\partial \psi}{\partial x}(0,0) + 2 \left( \frac{\partial \psi}{\partial \mu}(0,0)\right) \left( \frac{\partial \psi}{\partial x}(0,0)\right)= 2 \left(\frac{\partial \zeta}{\partial x}(0,0,0)\right) \left(\frac{\partial \zeta}{\partial \mu}(0,0,0)\right).
	\end{equation}
	Finally, differentiating both sides of \cref{eq:transcriticalcanonicalmain} twice with respect to $\mu$ and considering \cref{eq:transcriticalcanonicaldhdmu}, we obtain
	\begin{equation}\label{eq:transcriticalcanonicaldmumu}
		2 \frac{\partial \psi}{\partial \mu}(0,0) + 2 \left(\frac{\partial \psi}{\partial \mu}(0,0)\right)^2 = 2 \left(\frac{\partial \zeta}{\partial \mu}(0,0,0)\right)^2  + \frac{\partial^2 h}{\partial \mu^2}(0,0).
	\end{equation}
	
	Squaring \cref{eq:transcriticalcanonicaldmux} and considering \cref{eq:transcriticalcanonicaldxx}, it follows that
	\begin{equation}\label{eq:transcriticalcanonical3}
		1 + 4 \frac{\partial \psi}{\partial \mu}(0,0) + 4 \left(\frac{\partial \psi}{\partial \mu}(0,0)\right)^2 = 4 \left(\frac{\partial \zeta}{\partial \mu}(0,0,0)\right)^2.
	\end{equation} 
	Hence, combining with \cref{eq:transcriticalcanonicaldmumu}, we obtain
	\begin{equation}\label{eq:transcriticalcanonicaldhdmumu}
		\frac{\partial^2 h}{\partial \mu^2} (0,0) = -\frac{1}{2}.
	\end{equation}
	
	Considering \cref{eq:transcriticalcanonicaldhdmu,eq:transcriticalcanonicaldhdmumu}, it follows from Taylor's theorem that $h(\mu,0) = \mu^2 r(\mu)$, where $r$ is smooth and $r(0)=-\frac{1}{4}<0$. Hence, it is clear that $[h_0] = [r] \cdot [s_{1,0}]$, and $[h_0]$ is $\mathcal{K}$-equivalent to $[s_{1,0}]$. Thus, as before, it follows that the 1-parameter unfolding $[h]$ of $[h_0]$ must be $\mathcal{K}$-induced by $([r],[\Id]) * [\tilde{H}]$, that is, there are smooth real functions $S(\mu,\e)$, $a(\mu,\e)$, and $b(\e)$, defined on $(-\tilde{\mu}_2,\tilde{\mu}_2) \times (-\tilde{\e}_2,\tilde{\e}_2) \subset (-\tilde{\mu}_1,\tilde{\mu}_1) \times (-\tilde{\e}_1,\tilde{\e}_1)$, such that $S(\mu,0)=r(\mu)$, $a(\mu,0)=\mu$, $b(0) = 0$, and
	\begin{equation}
		h(\mu,\e) = S(\mu,\e) \cdot \left(a^2(\mu,\e) + b(\e)\right).
	\end{equation}
	holds locally near the origin. Since $S(0,0)=r(0)<0$, we can assume that $\tilde{\mu}_2$ and $\tilde{\e}_2$ are sufficiently small as to ensure that $S(\mu,\e)<0$ for any $(\mu,\e) \in (-\tilde{\mu}_2,\tilde{\mu}_2) \times (-\tilde{\e}_2,\tilde{\e}_2)$. Moreover, they can be assumed sufficiently small to guarantee that $a_\e$ is a diffeomorphism as well.
\end{proof}

\subsubsection{Proof of \cref{theoremtranscriticalcatastrophe}}

By definition of $\Delta_\ell$, it is easy to see that 
\begin{equation}
	\frac{\partial \Delta_\ell}{\partial \e}(0,0,0) = g_{\ell+1}(0,0),
\end{equation}
which is non-zero by hypothesis. Let $V : = I \times U_\Sigma \times (-\e_1,\e_1)$ as given in \cref{propositiontransn1}. Then, \cref{propertytransmain} ensures that 
\begin{equation}
	\frac{\partial \Delta_\ell}{\partial \e}(0,0,0) = Q(0,0,0) S(0,0) b'(0).
\end{equation}
Thus, considering \cref{propertytranseq,propertytransS}, it follows that 
\begin{equation}\label{eq:transcriticalproof1}
	b'(0) = \sigma \frac{g_{\ell+1}(0,0)}{|Q(0,0,0) S(0,0)|},
\end{equation}
where 
\begin{equation}
	\sigma = \sgn \left(\frac{\partial^2 g_\ell}{\partial x^2}(0,0)\right) \in \{-1,1\}.
\end{equation}

Now, \cref{propertytransmain} also ensures that $\Delta_\ell(x,\mu,\e)=0$ is equivalent to 
\begin{equation}\label{eq:transcriticaldeltazero}
	b(\e) =  -\frac{1}{S(\mu,\e)} \zeta^2(x,\mu,\e) - a^2 (\mu,\e)
\end{equation}
in $V$.
Define $\Psi(x,\mu,\e) = (\Psi_1(x,\mu,\e),\Psi_2(\mu,\e),\Psi_3(\e))$ by
\begin{equation}
	\Psi_1(x,\mu,\e) = \frac{\zeta(x,\mu,\e)}{\sqrt{-S(\mu,\e)}}, \quad \Psi_2(\mu,\e) = a(\mu,\e), \quad \Psi_3(\e) = b(\e).
\end{equation}
Hence, $\Psi$ is a strongly-fibred diffeomorphism onto its image $U$ and \cref{eq:transcriticaldeltazero} is itself equivalent to
\begin{equation}
	\Psi_3(\e) =  \left(\Psi_1(x,\mu,\e)\right)^2 - \left(\Psi_2(\mu,\e)\right)^2.
\end{equation}
Thus, $\Delta_\ell(x,\mu,\e) =0 \iff \Psi(x,\mu,\e) \in \left\{(y,\theta,\eta) \in \R^3: \eta = y^2-\theta^2\right\}$. Defining $\Phi = \Psi^{-1}$, it follows that $\Delta_\ell(x,\mu,\e) = 0 \iff (x,\mu,\e) \in \Phi\left(\left\{(y,\theta,\eta) \in \R^3: \eta = y^2-\theta^2\right\}\right) $.

Furthermore, since $\Phi_3(\e) = b^{-1}(\e)$, it follows from \cref{eq:transcriticalproof1} that
\begin{equation}
	\sgn \left(\Phi_3'(0)\right) = \sigma \cdot \sgn\left(g_{\ell+1}(0,0)\right).
\end{equation}
Finally, since $\Delta_\ell(x,\mu,0) = T g_\ell(x,\mu)$, it is easy to see that, if we fix $\e=0$, we have $g_\ell(x,\mu) =0 \iff \left(\Psi_1(x,\mu,0)\right)^2 = \left(\Psi_2(\mu,0)\right)^2$, proving that 
\begin{equation}
	\left(Z_{g_\ell} \times \{0\}\right) \cap V = \Phi\left(\left\{(y,\theta,0) \in \R^3 : y^2-\theta^2 = 0 \right\} \cap U\right).
\end{equation}

\subsubsection{Description of the perturbed bifurcation}
We now make use of the results above to describe the behaviour of $\Pi$ for values of the parameter near the point of bifurcation. Essentially, we show that, in one direction of variation of $\e$, the transcritical is broken into two nearby folds, whereas in the other no bifurcation occurs.

We assume, without loss of generality, that 
\begin{equation}
	\sgn\left(\frac{\partial^2 g_\ell}{\partial x^2}(0,0)\right) \, \sgn \left(g_{\ell+1}(0,0)\right)=1,
\end{equation}
which is equivalent to assuming the orientation of the saddle obtained for the catastrophe surface in \cref{theoremtranscriticalcatastrophe}. If this product is negative, the behaviour is analogous, but mirrored with respect to the sign of the perturbation parameter $\e$.
\begin{proposition} \label{propositiontranscriticalbreakage1}
	Let $n=1$ and suppose the guiding system $\dot x = g_\ell (x,\mu)$ undergoes a transcritical bifurcation at the origin for $\mu=0$.
	Also, let $I$, $U_\Sigma$ and $\e_1$ be as provided in \cref{propositiontransn1}, and define $\sigma=\sgn \left(\frac{\partial^2 g_\ell}{\partial x^2}(0,0) \right)$ and $\sigma' = \sgn \left(g_{\ell+1}(0,0)\right)$.
	If $\sigma \sigma' =1$, there are $(x_2,\mu_2,\e_2) \in (I\cap\R^*_+) \times (U_\Sigma \cap \R^*_+) \times (0,\e_1)$ and continuous functions  $\mu_c,\mu_e:(-\e_2,\e_2) \to (-x_2,x_2)$ such that the following hold:
	\begin{enumerate}[label=(\alph*)]
		\item\label{propertytransbreakagea} For each $\e \in (-\e_2,0)$, the family $(x,\mu) \mapsto \Pi(x,\mu,\e)$ undergoes two fold-like bifurcations in the set $(-x_2,x_2)$ as $\mu$ traverses $(-\mu_2,\mu_2)$, one at $\mu=\mu_e(\e)\in(0,\mu_2)$ and another at $\mu=\mu_c(\e)\in(-\mu_2,0)$. In other words, if we take $\mu$ to grow through $(-\mu_2,\mu_2)$, we observe the collision of two hyperbolic fixed points as $\mu=\mu_c(\e)$ and the subsequent emergence of two hyperbolic fixed points at $\mu=\mu_e(\e)$. When $\mu=\mu_c(\e)$ or $\mu=\mu_e(\e)$, there is one fixed point that is nonhyperbolic. Apart from those mentioned, there are no other fixed points in the interval $(-x_2,x_2)$. In particular, there are no fixed points in this interval for $\mu \in (-\mu_c(\e),\mu_e(\e))$.
		\item \label{propertytransbreakageb} For each $\e \in (0,\e_2)$, the family $(x,\mu)\mapsto \Pi(x,\mu,\e)$ does not undergo any bifurcation in $(-x_2,x_2)$ as $\mu$ traverses $(-\mu_2,\mu_2)$. If we take $\mu$ to grow past this interval, we observe exactly two hyperbolic fixed points in $(-x_2,x_2)$, first approaching without colliding, and then straying apart.
	\end{enumerate} 
\end{proposition}
\begin{proof}
	Take $\tilde{\e}_1:=\e_1$, $\tilde{x}_1,\tilde{\mu}_1>0$ such that $(-\tilde{x}_1,\tilde{x}_1) \subset I$, and $[-\tilde{\mu}_1,\tilde{\mu}_1] \subset U_\Sigma$ and define $W_1 = (-\tilde{x}_1,\tilde{x}_1) \times (-\tilde{\mu}_1,\tilde{\mu}_1) \times (-\tilde{\e}_1,\tilde{\e}_1)$. In that case, \cref{propertytransmain} ensures that
	\begin{equation}\label{eq:transcriticalcanonicalbreakage}
		\Delta_\ell(x,\mu,\e) = Q(x,\mu,\e)\left( \zeta^2(x,\mu,\e) + S(\mu,\e) a^2(\mu,\e)  + S(\mu,\e) b(\e) \right),
	\end{equation}
	for $(x,\mu,\e) \in W_1$.
	
	Let $\Lambda:(-\tilde{\mu}_1,\tilde{\mu}_1) \times (-\tilde{\e}_1,\tilde{\e}_1) \to \R^2$ be given by $\Lambda(\mu,\e) = (a(\mu,\e),\e)$. Since $a_\e$ is a diffeomorphism on $U_\Sigma$ for $ \e \in (-\tilde{\e}_1,\tilde{\e}_1)$, it follows that $\Lambda$ is a diffeomorphism onto its image. Considering that $a(0,0)=0$, $E_\Lambda: = \Im \Lambda$ is an open set containing $(0,0) \in \R^2$. Thus, there is a basic open neighbourhood of the origin $(-\tilde{a},\tilde{a}) \times (-\tilde{\e}_2,\tilde{\e}_2)\subset E_\Lambda$. This means that $(-\tilde{a},\tilde{a}) \subset \Im a_\e$ for any $\e \in  (-\tilde{\e}_2,\tilde{\e}_2)$. Since $b(0)=0$ and $b$ is smooth, we can take $\tilde{\e}_3 \in (0,\tilde{\e}_2)$ such that $\sqrt{|b(\e)|}<\tilde{a}$ for any $\e \in (-\tilde{\e}_3,\tilde{\e}_3)$. This ensures that $a_\e^{-1} \left(\pm \sqrt{|b(\e)|}\right)$ is well defined for $\e \in (-\tilde{\e}_3,\tilde{\e}_3)$. Hence, we can define 
	\begin{equation}
		\mu_c(\e):=a_\e^{-1} \left(-\sqrt{|b(\e)|}\right) \quad \text{and} \quad \mu_e(\e):= a_\e^{-1} \left(\sqrt{|b(\e)|}\right),
	\end{equation}
	both clearly continuous on $(-\tilde{\e}_3,\tilde{\e}_3)$ and whose image lies in $(-\tilde{\mu}_1,\tilde{\mu}_1)$.
	
	Proceeding as in the proof of \cref{theoremtranscriticalcatastrophe}, we obtain
	\begin{equation}
		b'(0) =  \sigma \frac{g_{\ell+1}(0,0)}{|Q(0,0,0) S(0,0)|},
	\end{equation}
	which does not vanish by hypothesis. Hence, there is $\tilde{\e}_4 \in (0,\tilde{\e}_3)$ sufficiently small such that $\sgn(b(\e)) = \sigma  \sigma'$ for $\e \in (0,\tilde{\e}_4)$ and $\sgn(b(\e))=-\sigma  \sigma'$ for $\e \in (-\tilde{\e}_4,0)$.
	Henceforth in the proof, we assume, without loss of generality, that $\sigma  \sigma' =1$. The other case can be treated analogously and will be omitted for the sake of brevity.
	
	Now, from \cref{eq:transcriticalcanonicalbreakage}, it follows that, for $(x,\mu,\e) \in W_1$, it holds that $\Delta_\ell(x,\mu,\e)=0$ if, and only if,
	\begin{equation}
		\left(\frac{\zeta(x,\mu,\e)}{\sqrt{-S(\mu,\e)}}\right)^2 = a^2 (\mu,\e) + b(\e).
	\end{equation}
	We will, therefore, study how many roots of the polynomial $z^2 = a^2(\mu,\e) +b(\e)$ exist near zero for each $(\mu,\e) \in (-\tilde{\mu}_1,\tilde{\mu}_1) \times (-\tilde{\e}_4,\tilde{\e}_4)$, since they can then be converted via inverse function to values of $x$ satisfying $\Delta_\ell(x,\mu,\e)=0$.

	We first study \cref{propertytransbreakagea}, that is, the case $\e \in (-\e_4,0)$, for which the polynomial equation can be rewritten as $z^2 = a^2(\mu,\e) + |b(\e)|$. Considering that $|b(\e)|>0$ for any $\e \in (-\tilde{\e}_4,0)$, it is easy to see that this equation has exactly two simple roots for $(\mu,\e) \in (-\tilde{\mu}_1,\tilde{\mu}_1) \times (-\tilde{\e}_4,0)$.
	
	Now, we consider \cref{propertytransbreakageb}, that is, the case $\e \in (0,\tilde{\e}_4)$, for which the polynomial equation can be rewritten as
	\begin{equation}
		z^2 = a^2(\mu,\e) - |b(\e)|.
	\end{equation}
	It is thus clear that this equation will have two simple real roots if $a^2(\mu,\e)>|b(\e)|$, one double real root if $a^2(\mu,\e)=|b(\e)|$ and no real roots if $a^2(\mu,\e)<|b(\e)|$. In other words, the number of roots depends solely on the sign of the function 
	\begin{equation}
		c_\e(\mu)= a^2(\mu,\e)-|b(\e)|.
	\end{equation} 
	There are, for each $\e \in(0,\tilde{\e}_4)$, exactly two values of $\mu \in (-\tilde{\mu}_1,\tilde{\mu}_1)$ for which $c_\e(\mu)=a^2(\mu,\e)-|b(\e)| = 0$, namely $\mu_c(\e)=a_\e^{-1} (-\sqrt{|b(\e)|})$ and $\mu_e(\e)=a_\e^{-1} (\sqrt{|b(\e)|})$. We proceed by studying the sign of the $c_\e(\mu)$ for $\mu \in (-\tilde{\mu}_1,\tilde{\mu}_1)$. 
	
	To do so, assume first that $a'_0(0)>0$. Since $a_\e$ is a diffeomorphism on $U_\Sigma$ for any $\e \in (-\tilde{\e}_1,\tilde{\e}_1)$, smoothness of $a$ ensures that $a'_\e(0)>0$ for any $\e \in (-\tilde{\e}_1,\tilde{\e}_1).$ For the same reason, we obtain that $a'_\e(\mu)>0$ for any $(\mu,\e) \in U_\Sigma \times (-\tilde{\e}_1,\tilde{\e}_1)$. Hence, since $[-\tilde{\mu}_1,\tilde{\mu}_1] \subset U_\Sigma$ and $[-\tilde{\e}_3,\tilde{\e}_3] \subset (-\tilde{\e}_1,\tilde{\e}_1)$, it follows that
	\begin{equation}
		m:=\inf \{a'_\e(\mu): (\mu,\e) \in [-\tilde{\mu}_1,\tilde{\mu}_1] \times (-\tilde{\e}_3,\tilde{\e}_3) \} >0.
	\end{equation}
	Moreover, considering that $(-\tilde{\mu}_1,\tilde{\mu}_1) \subset [-\tilde{\mu}_1,\tilde{\mu}_1]$ and that $(-\tilde{\e}_4,\tilde{\e}_4) \subset [-\tilde{\e}_3,\tilde{\e}_3]$, we get
	\begin{equation}
		\inf \{a'_\e(\mu): (\mu,\e) \in (-\tilde{\mu}_1,\tilde{\mu}_1) \times (-\tilde{\e}_4,\tilde{\e}_4) \} \geq m >0.
	\end{equation}
	This means that $a_\e$, and consequently also its inverse, is an strictly increasing function on $(-\tilde{\mu}_1,\tilde{\mu}_1)$, which will allow us to fully understand the sign of $c_\e$.
	
	Firstly, since $-\sqrt{|b(\e)|} <0 < \sqrt{|b(\e)|}$ and $a_\e^{-1}$ is increasing for any $\e \in (0,\tilde{\e}_4)$, it follows that
	\begin{equation}
		\mu_c(\e)=a_\e^{-1} \left(-\sqrt{|b(\e)|}\right)<a_\e^{-1}(0) < a_\e^{-1} \left(\sqrt{|b(\e)|}\right) = \mu_e(\e).
	\end{equation}
	Thus, since $a_\e$ is also increasing for any $\e \in (0,\tilde{\e}_4)$, we obtain
	\begin{equation}
		a_\e(\mu_c(\e)) < 0 < a_\e(\mu_e(\e)).
	\end{equation}
	Therefore, considering that $c_\e'(\mu) = 2a'_\e(\mu) a_\e(\mu)$ and that $a'_\e>0$, we conclude that
	\begin{equation}
		c_\e'(\mu_c(\e)) <0 < c_\e'(\mu_e(\e)),
	\end{equation}
	for any $\e \in (0,\tilde{\e}_4)$.
	
	We have thus proved that, if $\mu \in (0,\tilde{\e}_4)$, then $c_\e(\mu)>0$ for $\mu \in (-\tilde{\mu}_1, \mu_c(\e)) \cup (\mu_e(\e),\tilde{\mu}_1)$ and $c_\e(\mu)<0$ for $\mu \in (\mu_c(\e),\mu_e(\e))$. As mentioned before, this suffices to prove \cref{propertytransbreakageb}.
	
\end{proof}

\subsection{Proof of stabilisation of non-stable families: the pitchfork case}

The pitchfork bifurcation for flows is a 1-parameter family of 1-dimensional vector fields exhibiting the emergence of three equilibria from one persistent one, with a normal form $\dot x = \mu x + x^3$. If this family is perturbed, this behaviour is generally lost, unless some symmetry is assumed for the perturbation term, and the pitchfork bifurcation appears generically of families with symmetry (the so-called $\mathbb{Z}_2$-equivariant systems - see \cite[Section 7.4.2]{Kuznetsov2023}, for instance).

Similar to \Cref{sec:nonstabtran}, we begin with a definition of the pitchfork bifurcation based on the concept of $\mathcal{K}$-equivalence, before considering what happens when it occurs in a guiding system.

\begin{definition} \label{def:pitchforkn1}
	A 1-parameter family of $1$-dimensional vector fields $F(x,\mu)$ is said to undergo a pitchfork bifurcation at the origin for $\mu=0$ if
	\begin{enumerate}
		\item The germ of $f: x \mapsto F(x,0)$ at the origin is $\mathcal{K}$-equivalent to the germ of $s_{1^2,0}(x)=x^3$.
		\item  Let $([M],[\phi]) \in GL_n(\mathcal{E}_n)\times L_n$ be such that $[f]=[M] \cdot [s_{1^2,0}] \circ [\phi]$. The pushforward $([M],[\phi]) *[\tilde{\mathcal{U}}]$ of the unfolding $[\mathcal{U}] \in \mathcal{Z}_{1,1}^1$, given by $\tilde{\mathcal{U}}(x,\mu)=(\mathcal{U}(x,\mu),\mu)$ and $\mathcal{U}(x,\mu)=\mu x + x^3$, is $\mathcal{K}$-equivalent to $ [\tilde{F}]$ via the identity, where  $\tilde{F}(x,\mu)=(F(x,\mu),\mu)$.
	\end{enumerate}
\end{definition}

\subsubsection{The canonical form of the displacement function}
\begin{proposition} \label{propositionpitchn1}
	Let $n=k=1$ and suppose that the guiding system $\dot x = g_\ell (x,\mu)$ undergoes a pitchfork bifurcation at the origin for $\mu=0$.
	Then, there are $\e_1 \in (0,\e_0)$, an open interval $I$ containing $0 \in \R$, an open neighbourhood $U_\Sigma \subset \Sigma$ of $0$, and smooth functions $\zeta,Q: I\times U_\Sigma \times (-\e_1,\e_1) \to \R$, $a: U_\Sigma \times (-\e_1,\e_1) \to \R$, and $b:(-\e_1,\e_1) \to \R$ such that
	\begin{enumerate}[label=(P.\Roman*)]
		\item \label{propertypitchmain} If $\Delta_\ell$ is the displacement function of order $\ell$ of \cref{eq:standardsystemintro}, then 
		\begin{align*}
			\Delta_\ell(x,\mu,\e) = Q(x,\mu,\e)\left( \zeta^3(x,\mu,\e) + a(\mu,\e) \zeta(x,\mu,\e)  + b(\mu,\e) \right)
		\end{align*}
		for $(x,\mu,\e) \in I \times U_\Sigma \times (-\e_1,\e_1)$.
		\item \label{propertypitchphi} For each $(\mu,\e) \in U_\Sigma \times (-\e_1,\e_1)$, the map $\zeta_{(\mu,\e)}:x \mapsto \zeta(x,\mu,\e)$ is a diffeomorphism on the interval $I$.
		\item \label{propertypitcha} For each $\e \in(-\e_1,\e_1)$, $a_\e: \mu \mapsto a(\mu,\e)$ is a diffeomorphism on $U_\Sigma$.
		\item \label{propertypitcheq}$b(0,0)= \frac{\partial b}{\partial \mu}(0,0)=0$, $a(0,0) =0$, $\zeta(0,0,0)=0$, and $Q(0,0,0) \neq 0$.
	\end{enumerate}  
\end{proposition}
\begin{proof}
	Observe that $\Delta_\ell(x,\mu,0) = T g_\ell(x,\mu)$, by definition. Let $[s_{1^2,0}]$ be as in \cref{def:pitchforkn1} and $[\tilde{F}]$ be the 2-parameter unfolding of $f: x \mapsto Tg_\ell(x,0)$ given by $\tilde{F}(x,\mu,\e) = (\Delta_\ell(x,\mu,\e),\mu,\e)$. Since $[g_\ell]$ undergoes a pitchfork bifurcation, there are $P(x,\mu) \in \R$ and $\psi(x,\mu) \in \R$ such that
	\begin{equation}\label{eq:pitchforkcanonical1}
		\Delta_\ell(x,\mu,0) = P(x,\mu) \left(\mu \psi(x,\mu) + \psi^3(x,\mu)\right)
	\end{equation}
	and, defining $M(x):=P(x,0)$, and $\phi(x):=\psi(x,0)$, it holds that $[f] = [M] \cdot [s_{1,0}] \circ [\phi]$. 
	
	Let $\tilde{H}(x,\theta,\eta)=(y^3+\theta y +\eta,\theta,\eta) \in \R \times \R \times \R$. Since the 2-parameter unfolding $[\tilde{H}]$ of $[s_{1^2,0}]$ is $\mathcal{K}$-versal, then $[\tilde{F}]$ must be $\mathcal{K}$-induced by $([M],[\phi]) * [\tilde{H}]$. Therefore, there is a neighbourhood $\tilde{V}_1 : = I \times (-\tilde{\mu}_1,\tilde{\mu}_1) \times (-\tilde{\e}_1,\tilde{\e}_1)$ of the origin in $\R^{1+1+1}$ and smooth functions $h_(\mu,\e) = (a(\mu,\e),b(\mu,\e)) \in \R^2$, $Q(x,\mu,\e) \in \R$, and $\zeta(x,\mu,\e) \in \R $ such that $h(0,0) = (0,0)$, $Q(x,0,0) = M(x) \neq 0$, $\zeta(x,0,0) = \phi(x)$, and
	\begin{equation}\label{eq:pitchforkcanonical2}
		\Delta_\ell(x,\mu,\e) = Q(x,\mu,\e) \cdot \left(\zeta^3(x,\mu,\e)+a(\mu,\e) \zeta(x,\mu,\e) + b(\mu,\e)\right).
	\end{equation}
	Considering that $\zeta(x,0,0)=\phi(x)$ is a local diffeomorphism, if we assume that $\tilde{\mu}_1$ and $\tilde{\e}_1$ are sufficiently small, we can ensure that $\zeta_{(\mu,\e)}:x \mapsto \zeta(x,\mu,\e)$ is a diffeomorphism on $I$ for any $(\mu,\e) \in (-\tilde{\mu}_1,\tilde{\mu}_1) \times (-\tilde{\e}_1,\tilde{\e}_1)$.
	
	Combining \cref{eq:pitchforkcanonical1,eq:pitchforkcanonical2}, we have
	\begin{equation} \label{eq:pitchforkcanonicalmain}
		P(x,\mu) \left(\mu \psi(x,\mu) + \psi^3(x,\mu)\right) = Q(x,\mu,0) \cdot \left(\zeta^3(x,\mu,0)+a(\mu,0) \zeta(x,\mu,0) + b(\mu,0)\right).
	\end{equation}
	
	Differentiating both sides of \cref{eq:pitchforkcanonicalmain} with respect to $\mu$ at the origin and considering that $\psi(0,0) = \zeta(0,0,0)= \phi(0) = 0$, it follows that
	\begin{equation}\label{eq:pitchforkcanonicaldhdmu}
		\frac{\partial b}{\partial \mu}(0,0) = 0.
	\end{equation}
	
	Now, differentiating both sides of \cref{eq:pitchforkcanonicalmain}, once with respect to $x$ and once with respect to $\mu$, at the origin and considering that $M(0)$ is invertible, we obtain
	\begin{equation}\label{eq:pitchforkcanonicaldmux}
		\frac{\partial a}{\partial \mu}(0,0) = \frac{\partial \psi}{\partial x}(0,0) = \phi'(0) \neq 0.
	\end{equation}
	We can assume that $\tilde{\mu}_2$ and $\tilde{\e}_2$ are sufficiently small as to ensure that $\frac{\partial a}{\partial \mu}(\mu,\e)\neq0$ for any $(\mu,\e) \in (-\tilde{\mu}_2,\tilde{\mu}_2) \times (-\tilde{\e}_2,\tilde{\e}_2)$, guaranteeing that $a_\e$ is a diffeomorphism.
\end{proof}

\subsubsection{Proof of \cref{theorempitchforkcatastrophe}}

By definition of $\Delta_\ell$, it is easy to see that 
\begin{equation}
	\frac{\partial \Delta_\ell}{\partial \e}(0,0,0) = g_{\ell+1}(0,0),
\end{equation}
which is does not vanish. Let $V : = I \times U_\Sigma \times (-\e_1,\e_1)$ as given in \cref{propositionpitchn1}. Then, \cref{propertypitchmain} ensures that 
\begin{equation}
	\frac{\partial \Delta_\ell}{\partial \e}(0,0,0) = Q(0,0,0) b'(0).
\end{equation}
Thus, considering \cref{propertypitcheq}, we obtain
\begin{equation}\label{eq:pitchforkproof1}
	\frac{\partial b}{\partial \e}(0,0) =  \frac{g_{\ell+1}(0,0)}{Q(0,0,0)}.
\end{equation}

\Cref{propertypitchmain} yields that $\Delta_\ell(x,\mu,\e)=0$ is equivalent to 
\begin{equation}\label{eq:pitchforkdeltazero}
	b(\mu,\e) =  -\zeta^3(x,\mu,\e) - a (\mu,\e) \zeta(x,\mu,\e)
\end{equation}
in $V$. Define $\Psi(x,\mu,\e) = (\Psi_1(x,\mu,\e),\Psi_2(\mu,\e),\Psi_3(\mu,\e))$ by
\begin{equation}
	\Psi_1(x,\mu,\e) = \zeta(x,\mu,\e), \quad \Psi_2(\mu,\e) = a(\mu,\e), \quad \Psi_3(\mu,\e) = b(\mu,\e),
\end{equation}
a weakly-fibred diffeomorphism onto its image $U$. We remark that, since $\frac{\partial b}{\partial \mu}(0,0)=0$ by \cref{propertypitcheq}, we also know that $\Psi$ is strongly-fibred to the first order at the origin. Moreover, \cref{eq:pitchforkdeltazero} is equivalent to
\begin{equation}
	\left(\Psi_1(x,\mu,\e)\right)^3 + \Psi_2(\mu,\e) \Psi_1(x,\mu,\e) + \Psi_3(\mu,\e) = 0. 
\end{equation}

Therefore, $\Delta_\ell(x,\mu,\e) =0 \iff \Psi(x,\mu,\e) \in \left\{(y,\theta,\eta) \in \R^3:y^3-\theta y + \eta =0\right\}$. Defining $\Phi = \Psi^{-1}$, the proof is concluded.

\subsection{Proof of \cref{theoremsaddlenode}}
	Let $\Phi$ be as in \cref{maintheorem} and $\e\neq 0 $ be small enough so that $\e' = \Phi^{-1}_3(\e)$ is well defined. By \cref{maintheorem}, the points $(x,\mu)$ near $(0,0)$ for which $\Pi(x,\mu,\e)=x$ are given by $(\Phi_1(\alpha(t),\eta(t),\e'),\Phi_2(\eta(t),\e'))$, where $(\alpha(t),\eta(t))$ are a local parametrisation near $(0,0)$ of the curve given by $g_\ell(\alpha,\eta)=0$. Considering the Implicit Function Theorem and \cref{saddlenode1,saddlenode2}, we can assume that $\alpha(t)=t$.
	
	Thus, since $g_\ell(t,\eta(t))=0$, by differentiating with respect to $t$ at $t=0$, we obtain
	\begin{equation}
		\frac{\partial g_\ell}{\partial \mu}(0,0) \eta'(0) =0,
	\end{equation}
	which ensures that $\eta'(0)=0$.
	
	Now, differentiating $\Pi(\Phi_1(t,\eta(t),\e'),\Phi_2(\eta(t),\e'),\e)=\Phi_1(t,\eta(t),\e')$ with respect to $t$ at $t=0$, it follows that
	\begin{equation}
		\frac{\partial \Pi}{\partial x_0}(x^*(\e),\mu^*(\e),\e) = 1,
	\end{equation}
	where $x^*(\e):=\Phi(0,0,\e')$ and $\mu^*(\e):=\Phi_2(0,\e')$.

	Taking into account \cite[Theorems 4.1 and 4.2]{Kuznetsov2023}, we need only verify the two genericity conditions that guarantee a fold up to topological conjugacy:
	\begin{enumerate}[label=(F\textsubscript{top}\arabic*)]
		\item $\frac{\partial \Pi}{\partial \mu}(x^*(\e),\mu^*(\e),\e)\neq 0$;
		\item $\frac{\partial^2 \Pi}{\partial x^2}(x^*(\e),\mu^*(\e),\e) \neq 0 $.
	\end{enumerate}
	These follow directly from smoothness with respect to $\e$, combined with \cref{eq:poincareelldisplacementintro}, the fact that $\Delta_\ell (x,\mu,0) = T g_\ell(x,\mu)$, and \cref{saddlenode1,saddlenode2}.

\section{Acknowledgements}
PCCRP is supported by S\~{a}o Paulo Research Foundation (FAPESP) grants nº 2023/11002-6 and 2020/14232-4. DDN was supported by the São Paulo Research Foundation (FAPESP), Grant No. 2024/15612-6; by the Conselho Nacional de Desenvolvimento Científico e Tecnológico (CNPq), Grant No. 301878/2025-0; and by the Coordenação de Aperfeiçoamento de Pessoal de Nível Superior - Brasil (CAPES), through the MATH-AmSud program, Grant No. 88881.179491/2025-01.

\appendix 

\section{Group structure of germs of fibred diffeomorphisms}\label{sec:fibred}
It is known that germs of local diffemorphisms at a point (see \cref{def:germslocaldiff}) have a well defined operation induced by composition. Hence, we assume without loss of generality that the domains and images of the diffeomorphisms are compatible with composition.

The fact that the composition of two fibred diffeomorphisms is still a fibred diffeomorphism, be it strongly or weakly fibred, amounts to simple calculation, and will be omitted here. The only property of groups that has to be non-trivially verified is the existence of an inverse element in the class of local diffeomorphisms with the same fibration, which amounts to proving that the inverse of a fibred diffeomorphism is itself still fibred. 

Let thus $\Phi$ be strongly-fibred and let $\Psi:=\Phi^{-1}$, its inverse diffeomorphism. We wish to prove that $\Psi$ is strongly-fibred as well. we begin by proving that $\Psi_3$ does not depend on $x$ or $\mu$.

To do so, first notice that, since $\Phi$ is diffeomorphism, it follows that, for any $(x,\mu,\e)$ in its domain, 
$\det D\Phi(x,\mu,\e) \neq0$. Considering that 
\begin{equation}
	\Phi(x,\mu,\e)=(\Phi_1(x,\mu,\e),\Phi_2(\mu,\e),\Phi_3(\e)),
\end{equation} 
it follows at once by taking into account the block structure of the matrix $D\Phi(x,\mu,\e)$ that
\begin{equation}\label{eq:appendix1}
	\det \frac{\partial \Phi_1}{\partial x}(x,\mu,\e) \neq 0, \quad 	\det \frac{\partial \Phi_2}{\partial \mu}(\mu,\e) \neq 0, \quad \text{and} \quad \Phi_3'(\e) \neq 0.
\end{equation}

Now differentiate the identity $\Psi_3 (\Phi_1(x,\mu,\e),\Phi_2(\mu,\e),\Phi_3(\e)) = \e$ with respect to $x$ to obtain
\begin{equation}
	\frac{\partial \Psi_3}{\partial x} (\Phi_1(x,\mu,\e),\Phi_2(\mu,\e),\Phi_3(\e)) \cdot \frac{\partial \Phi_1}{\partial x} (x,\mu,\e) = 0,
\end{equation}
which, combined with \cref{eq:appendix1}, ensures that $\frac{\partial \Psi_3}{\partial x}$ vanishes identically in its domain. 

Differentiating the same identity with respect to $\mu$ and considering that $\frac{\partial \Psi_3}{\partial x}=0$, we obtain 
\begin{equation}
	\frac{\partial \Psi_3}{\partial \mu} (\Phi_1(x,\mu,\e),\Phi_2(\mu,\e),\Phi_3(\e)) \cdot \frac{\partial \Phi_2}{\partial \mu} (\mu,\e) = 0 ,
\end{equation}
now ensuring that $\frac{\partial \Psi_3}{\partial x}$ vanishes identically in its domain. Therefore, $\Psi_3$ depends solely on $\e$, as we wished to prove. 

Finally, differentiating $\Psi_2 (\Phi_1(x,\mu,\e),\Phi_2(\mu,\e),\Phi_3(\e)) = \mu$ with respect to $x$, it follows that
\begin{equation}
	\frac{\partial \Psi_2}{\partial x} (\Phi_1(x,\mu,\e),\Phi_2(\mu,\e),\Phi_3(\e)) \cdot \frac{\partial \Phi_1}{\partial x} (x,\mu,\e) = 0,
\end{equation}
which proves that $\Psi_2$ is independent of $x$, finishing the proof for the strongly-fibred case.

The weakly-fibred case is analogous, and so will be omitted.

\bibliographystyle{plain}
\bibliography{newrefs}

\end{document}